\documentclass[11pt]{amsart}
\usepackage[foot]{amsaddr}   
\usepackage{graphicx, amssymb,amsmath, latexsym,cite,fancyhdr}
\usepackage[dvipsnames]{xcolor} 
\usepackage{amsthm,arydshln}
\usepackage{multirow}     
\usepackage{colortbl}    
\usepackage{booktabs} 
\usepackage{times}       
\usepackage[T1]{fontenc}      
\usepackage{url}
\usepackage{mathabx,xfrac} 
\usepackage{ifthen}  
\usepackage{longtable}   

\usepackage{tikz-cd}

\usepackage{graphicx} 
\usepackage{booktabs} 
\usepackage{xparse}   
\NewDocumentCommand{\rot}{O{30} O{1em} m}{\makebox[#2][l]{\rotatebox{#1}{#3}}}%

\usepackage[ruled,vlined]{algorithm2e}

\newtheoremstyle{mythm} 
{6pt}
{6pt}
{\it}
{}
{\bf}
{.}
{.5em}
{}
 
\newtheoremstyle{mydef}
{6pt}
{6pt}
{}
{}
{\bf}
{.}
{.5em}
{}

\newtheoremstyle{myrem}
{6pt}
{6pt}
{}
{}
{\bf}
{.}
{.5em}
{}

\newtheoremstyle{myex}
{6pt}
{3ex}
{}
{}
{\bf}
{.}
{.5em}
{}

\theoremstyle{mythm}
\newtheorem{theorem}{Theorem}[section]
\newtheorem{lemma}[theorem]{Lemma}
\newtheorem{proposition}[theorem]{Proposition}

\theoremstyle{mydef}
\newtheorem{definition}[theorem]{Definition}

\theoremstyle{myrem}

\newtheorem{notation}[theorem]{Notation}
\numberwithin{equation}{section}

\theoremstyle{myex}
\newtheorem{example}[theorem]{Example}

\usepackage[a4paper,right=2.3cm, left=2.3cm, top=2.5cm, bottom=2.2cm, marginpar=1.6cm]{geometry}

\newcommand{\Hom}{{\rm Hom}}

\newcommand{\mymat}[4]{\left(\begin{smallmatrix}#1&#2\\#3&#4\end{smallmatrix}\right)}

\newcounter{ithmcount}

\newenvironment{myitemize}{
\begin{list}{$\bullet$}
{\labelwidth30pt \leftmargin12pt \topsep3pt \itemsep0pt \parsep6pt}}
{\end{list}}               

\newenvironment{iprf}{\begin{list}{{\rm
	\alph{ithmcount})}}{\usecounter{ithmcount}\labelwidth-5pt
      \leftmargin0pt \topsep3pt \itemsep1pt \parsep2pt}}{\qedhere\end{list}}

\newenvironment{ithm}{\begin{list}{{\rm \alph{ithmcount})}}{\usecounter{ithmcount}\labelwidth18pt
      \leftmargin18pt \topsep3pt \itemsep1pt \parsep4pt}}{\end{list}}

\allowdisplaybreaks

\newcommand{\ra}{\rightarrow}

\newcommand{\ol}{\overline}

\newcommand{\N}{\mathbb N}

\newcommand{\Z}{\mathbb Z}

\newcommand{\GL}{\mathrm{GL}}
\newcommand{\Aut}{\mathrm{Aut}}
\newcommand{\Comp}{\mathrm{CP}}
\newcommand{\Stab}{\mathrm{Stab}}

\newcommand{\ind}{\mathrm{ind}}
\newcommand{\DC}{\mathrm{DC}}
\newcommand{\GN}{\mathcal N}
\newcommand{\w}{\Delta}
\newcommand{\myw}[2]{\w_{#1}^{#2}}
\newcommand{\mydelta}[2]{\w_{#2}^{#1}}   
\newcommand{\Irr}{I}

\newcommand{\diag}{{\rm diag}}

\renewcommand{\O}{\mathcal X}

\renewcommand{\SS}{\mathcal S}
\newcommand{\TT}{\mathcal T}

\newcommand{\K}{\mathcal K}

\renewcommand{\leq}{\leqslant}
\renewcommand{\geq}{\geqslant}

\begin{document}

 \vspace*{-0.5cm}

\title{Groups whose orders factorise into at most four primes}
 
\author[H.~Dietrich]{Heiko Dietrich}
\email{heiko.dietrich@monash.edu}
\address[H.~Dietrich and X.~Pan]{School of Mathematics, Monash University, VIC 3800, Australia}
\author[B.~Eick]{Bettina Eick}
\email{beick@tu-bs.de}
\address[B.~Eick]{Institut Computational Mathematics, Technische Universit\"at Braunschweig, Germany}
\author[X.\ Pan]{Xueyu Pan}
\email{xpan.eileen@gmail.com}
\thanks{The first author was supported by an ARC grant DP190100317. The third author was supported by an RTP scholarship; part of this work is based on the MPhil thesis \cite{xp} of the third author. The authors thank Mike Newman  for various helpful discussions.} 

\begin{abstract}
The groups whose orders factorise into at most four  primes have been described (up to isomorphism) in various papers. Given such an order  $n$, this paper exhibits a new explicit and compact determination of the isomorphism types of the groups of order $n$ together with effective algorithms to enumerate, construct, and identify these groups. The algorithms are implemented for the computer algebra system GAP.
\end{abstract} 

\maketitle

\vspace*{-0.5cm}

\vspace*{-5.5cm}

\begin{minipage}{15.5cm}
\small \color{red} For the published version please see J.\ Symb.\ Comp.\ 108, 23-40 (2022). This version corrects two minor and obvious typos in Table 4, Cluster 6 and 7 (the red entries). Thanks to Bill Unger (USyd) for pointing these out.
\end{minipage}

\vspace*{5.5cm}

\section{Introduction}\label{s:1}
\noindent It is a major theme in group theory to enumerate and construct all groups of a given order up to isomorphism.  This has been initiated by Cayley (1854) who determined the groups of order at most $6$. Many
publications have followed Cayley's work; a survey is given in Besche-Eick-O'Brien \cite{BEO02}. The enumeration of the (isomorphism types of) groups of order $n$ is a related
problem: The number $\GN(n)$ of (isomorphism types of) groups of order $n$ is 
known for all $n\leq 2000$, see \cite{BEO02}, and for most $n$ at 
most $20000$, see Eick-Horn-Hulpke \cite{EHH16}, but no closed formula is known for general $n$. 
 We refer to the book of Blackburn-Neumann-Venkataraman \cite{BNV07} and the papers by Conway-Dietrich-O'Brien \cite{CDOB08} and Eick-Moede \cite{EM} for  more  information on group enumeration.

Modern computer algebra systems such as GAP \cite{gap} and Magma \cite{magma} contain (different though overlapping parts of) a database of groups of ``small order'': the SmallGroups Library. Given a ``small'' order $n$, this library contains a list of isomorphism type representatives for the groups of order~$n$. The library also contains an identification function for these groups: given a group $G$, one can determine its ID $(n,i)$, meaning that $G$ is isomorphic to the $i$-th group in the database list of groups of order $n$. The SmallGroups Library has become highly popular in the research community. It is therefore an important topic in computational group theory to extend this library.

\thispagestyle{empty}
\enlargethispage{5ex}

Our aim here is to present a new extension to the SmallGroups Library in form of the GAP package {\tt SOTGrps} \cite{xp}: we describe effective construction and identification methods for the groups of order $n$, where $n$ factorises into at most four primes. This is underpinned by a new explicit enumeration and a new compact determination of the groups of such orders.

If $n$ factorises into at most four primes, then its factorisation is of one of the following types: $p$, $pq$,$p^2$, $pqr$, $p^2q$, $p^3$, $pqrs$, $p^2qr$, $p^2q^2$, $p^3q$, and $p^4$, where $p,q,r,s$ are distinct primes. The groups of the square-free orders $p$, $pq$, $pqr$, and $pqrs$ and those of prime-power orders $p, p^2, p^3$ and $p^4$ have been determined by H\"older \cite{Hol93, Hol95b} and this is included in the SmallGroups library already. We refer to Slattery \cite{sf} for a recent discussion of the groups of square-free orders and to Dietrich-Low \cite{cgroups} for a generalisation of \cite{sf}. For a recent and accessible determination of the groups of order dividing $p^5$ we refer to Girnat \cite{Gir03}. It remains to consider the orders in
\[\mathcal{O}=\{p^2 q,\ p^3 q,\ p^2 q^2,\ p^2 q r \;:\; p,q,r \text{ distinct primes}\}.\]
These orders have also already been considered in the past:  Cole \& Glover \cite{cole} and H\"older \cite{Hol93} determined the groups of order  $p^2q$,  Western \cite{Wes99} those of order $p^3q$, Le Vavasseur \cite{Vav99, Vav02} and Lin \cite{Lin74} those of order $p^2q^2$,  and Glenn \cite{Gle06} those of order  $p^2 qr$.  Moreover, Laue \cite{Lau82} considered all orders $p^a q^b$ with $a+b \leq 6$ and $a,b<5$. In an arXiv article, Eick \cite{Ear} provided a compact and uniform enumeration of the orders in~$\mathcal{O}$. Our results here include this and extend it to a full determination. Since many of the orders in $\mathcal{O}$ are cube-free, we mention that Dietrich-Eick \cite{DEi05} and Dietrich-Wilson \cite{cf} developed a construction algorithm and isomorphism test for groups of cube-free order; however, these works do not naturally lead to efficient identification functionality, cf.\ Section \ref{secalg}.

We proceed as follows. In Section \ref{secMR} we describe our main results. In Section \ref{secprel} we prove our enumeration; a proof of the determination follows the same arguments. Our algorithms are discussed in Section~\ref{secconst}.

\section{Main results}\label{secMR}
\noindent For every  $n\in\mathcal{O}$ we provide an explicit closed formula for the number $\GN(n)$ of groups of isomorphism types of order $n$, see Section \ref{secmrenum}; each of these formulas is a generalisation of a PORC function (a polynomial on residue classes), that is, there are finitely many sets of number-theoretic  conditions on the involved primes so that $\GN(n)$ is a polynomial in the  involved primes for each of the condition-sets. The proof of our enumeration translates to new explicit group presentations, which we present in Section \ref{secpres}. These, in turn, lead to new efficient construction and identification algorithms, see Section \ref{secalgmr}.

\subsection{Enumeration}\label{secmrenum}

For integers $u,v\in \N$ the \emph{divisibility Kronecker delta} is \[\w_u^v=\begin{cases}1 &(\text{if }v\mid u)\\ 0 & (\text{otherwise.})
\end{cases}\]

\medskip

\noindent Our first main result is the following theorem; we prove it in  Sections \ref{pf1}--\ref{pf2}. 

\smallskip

\begin{theorem} \label{thp2q}\label{thp3q}\label{thp2qr}\label{thp2q2}
  Let $p$, $q$, and $r$ be distinct primes.
  \begin{ithm}
    
  \item {\bf Order $p^2q$:}
    \begin{myitemize}
    \item $\GN(p^2 q) = 5$ for $q = 2$.
    \item $\GN(p^2 q) = 2 + \tfrac{q+5}{2} \w_{p-1}^q + \w_{p+1}^q 
                        + 2 \w_{q-1}^p + \w_{q-1}^{p^2}$ for $q > 2$.
    \end{myitemize}

  \item {\bf Order $p^3q$:}
   \begin{myitemize}
\item There are two special cases $\GN(2^3.3) = 15$ and $\GN(2^3.7) = 13$.
\item $\GN(p^3q) = 15$ if $q=2$.
\item $\GN(p^3 q) = 12 + 2 \w_{q-1}^4 + \w_{q-1}^8$ if $p=2$ and 
   $q\notin\{3,7\}$.
\item If $p$ and $q$ are both odd, then
\begin{eqnarray*}
\GN(p^3 q) &=& 5 + \tfrac{q^2 + 13 q + 36}{6} \w_{p-1}^q  + (p+5)\w_{q-1}^p  + \tfrac{2}{3}\w_{q-1}^3 \w_{p-1}^q  
              \\&& + \w_{(p+1)(p^2+p+1)}^q(1- \w_{p-1}^q)  + \w_{p+1}^q + 2 \w_{q-1}^{p^2} + \w_{q-1}^{p^3}.
\end{eqnarray*}
\end{myitemize}

  \item {\bf Order $p^2q^2$ with $p>q$:}
  \begin{myitemize}
  \item There is one special case $\GN(3^2. 2^2) = 14$.
  \item $\GN(p^2q^2) = 12+4\w_{p-1}^4$ if $q = 2$ and $p\ne 3$.
  \item $\GN(p^2 q^2) = 4 + \tfrac{1}{2}(q^2+q+4) \w_{p-1}^{q^2}  + (q+6) \w_{p-1}^q + 2\w_{p+1}^q + \w_{p+1}^{q^2}$ if $q > 2$.
  \end{myitemize}

\item  {\bf Order $p^2qr$ with $q<r$:}
  \begin{myitemize}
  \item There is one special case $\GN(2^2. 3. 5 ) = 13$.
  \item $\GN(p^2qr) =  10   + (2r+7) \w_{p-1}^r    + 3 \w_{p+1}^r    + 6 \w_{r-1}^p    + 2 \w_{r-1}^{p^2}$ if $q=2$.
 \item If $q>2$ and $(p,q,r)\ne (2,3,5)$, then 
\begin{eqnarray*}
  \GN(p^2qr) &=&  2 \;+\; (p^2-p) \w_{q-1}^{p^2} \w_{r-1}^{p^2} \;+\; \w_{p+1}^r \;+\; \w_{p+1}^q\;+\;\w_{r-1}^{p^2}   \;+\; \w_{q-1}^{p^2} \\[0.5ex]
   && \;+\; (p-1) \big(\w_{q-1}^{p^2} \w_{r-1}^{p}
      \;+\;\w_{r-1}^{p^2} \w_{q-1}^p\;+\;2 \w_{r-1}^p \w_{q-1}^p\big) \\[0.5ex]
   && \;+\; \w_{r-1}^p \w_{q-1}^p \;+\; \w_{r-1}^p \w_{p-1}^q\;+\; \tfrac{1}{2}(q-1)(q+4) \w_{p-1}^q \w_{r-1}^q  \\[0.5ex]
   && \;+\; \tfrac{1}{2}(q-1)\big(\w_{p+1}^q \w_{r-1}^q \;+\; \w_{p-1}^q 
           \;+\; \w_{p-1}^{qr} \;+\; 2 \w_{r-1}^{pq} \w_{p-1}^q \big) \\[0.5ex]
   &&   \;+\; \tfrac{1}{2}(qr+1) \w_{p-1}^{qr}   \;+\; \tfrac{1}{2}(r+5) \w_{p-1}^r (1 + \w_{p-1}^q) \;+\; 2 \w_{r-1}^q\\[0.5ex]
   &&   \;+\; \w_{p^2-1}^{qr} \;+\; 2 \w_{r-1}^{pq}  
        \;+\; \w_{r-1}^{p^2q} \;+\; 2 \w_{q-1}^p  \;+\; 3 \w_{p-1}^q \;+\; 2 \w_{r-1}^p.
\end{eqnarray*}
  \end{myitemize}
  
  \end{ithm}
\end{theorem}

\subsection{Presentations}\label{secpres} The next theorem follows from our proof of Theorem \ref{thp2q} and is discussed in Section~\ref{secconst}.

\begin{theorem}
Presentations for the groups of order $n\in\mathcal{O}$, up to isomorphism, are given in Tables \ref{tablep2q}--\ref{tabp2qr}. The groups listed in the left column only exist if the number in the right column is greater than $0$. 
\end{theorem}

Throughout the paper, $p,q,r$ denote distinct primes, and $C_n^m$ denotes the $m$-fold direct product of the cyclic group $C_n$ of order  $n$. We write $\Z_n=\{0,\ldots,n-1\}$ for the ring of integers modulo $n$, and $\Z_n^\ast$ for its group of units.   The general linear group of $n\times n$ matrices over the field $\Z_p$ is denoted $\GL_n(p)$; we write $\diag(a_1,\ldots,a_n)$ for the diagonal matrix with diagonal entries $a_1,\ldots,a_n$.

\begin{notation}\label{not1}
  All solvable groups in Tables \ref{tablep2q}--\ref{tabp2qr} are specified by a polycyclic presentation, with the following convention: first, we only list the power relations and non-trivial commutator relations, so that the set of generators is implicitly defined by the power relations: e.g.\ $\{a^{p},b^p\}$ defines $\langle a,b\mid a^p, b^p, a^b=a\rangle\cong C_p^2$. Secondly, we abbreviate a relation $a=b$ by $a/b$, and if $a$ acts on $b,c\in C_p^2$ via a matrix $M$, then we specify these conjugacy relations by $(b^a,c^a)/(b,c)^M$. For example, $b^a=c$ and $c^a=b^2c^{3}$ is denoted by \[(b^a,c^a)/(b,c)^M\] where $M=\left(\begin{smallmatrix}0&1\\2&3\end{smallmatrix}\right)$. The remaining notation is as follows.

\medskip    
    
\begin{myitemize}
\item Throughout, $P$ is a Sylow $p$-subgroup, $Q$ is a Sylow $q$-subgroup, and $F$ is the Fitting subgroup.      
\item If $\Z_a^\ast$ is cyclic, then the \emph{canonical generator} $\sigma_a\in \Z_a^\ast$ is the smallest generator in $\{1,\ldots,a-1\}$. If $b$ divides $\phi(a)=|\Z_a^\ast|$ and $k$ is an integer, then we define \[\rho(a,b,k)=\sigma_a^{k\phi(a)/b}\quad\text{and}\quad \rho(a,b)=\rho(a,b,1).\]
\item For $k\in\{2,3\}$ fix a generator $\tau\in {\rm GF}(p^k)^\ast$. If $c\mid (p^k-1)$ but $c\nmid (p^{k-1}-1)$, then let $\gamma=\tau^{(p^k-1)/c}$ and define $\Irr_k(p,c)\in\GL_k(p)$ to be the companion matrix of the characteristic polynomial of $\diag(\gamma,\gamma^p,\ldots,\gamma^{p^{k-1}})$; this matrix generates a subgroup of order $c$ of a Singer cycle in $\GL_k(p)$. For example, if $p=5$, $k=2$, and $c=3$, then $I_2(5,3)=\left(\begin{smallmatrix}0&4\\1&4\end{smallmatrix}\right)$. 
\item If $c\mid (p-1)$  such that $\Z_c^\ast$ is cyclic, then for an integer $k$ we define 
  \[M(p,c,k)=\diag(\sigma_p^{(p-1)/c},\sigma_p^{k(p-1)/c})\quad\text{and}\quad M(p,c)=M(p,c,1).\]
\end{myitemize}
\end{notation}

\noindent For example, if $p=7$ and $q=3$, then, using the above conventions, $\{a^q, b^{p^2}, c^p, b^a / b^{\rho(p^2,q)}\}$ encodes the group
\[\langle a,b, c \mid a^3, b^{49}, c^7, b^a = b^{30}, c^a=c, c^b=c \rangle;\]
note that $\sigma_{7^2}=3$, so $\rho(p^2,q)=\sigma_{7^2}^{\phi(7^2)/3}=3^{14}\bmod 49=30$.

\enlargethispage{1ex}

\vspace*{0.5cm}

\newcommand{\mycol}{\cellcolor{lightgray!45}}
\newcommand{\myrcol}{\rowcolor{lightgray!45}}

{\footnotesize
  \renewcommand\arraystretch{1.1}
\begin{longtable}[ht]{lcr}  \toprule
  \multicolumn{3}{c}{\bf Groups of order $p^2q$  using Notation \ref{not1}.}\\
\toprule 
PC-relators                                            & Parameters                    & Number of groups                      \\ 
\hline
\multicolumn{3}{l}{{\mycol}{\bf Cluster 1:} nilpotent}                                                               \\
$a^{p^2q}$                                             &                               & $1$                                  \\
$a^{pq}, b^p$                                          &                               & $1$                                  \\ 
\hline
\multicolumn{3}{l}{{\mycol}{\bf Cluster 2:} non-nilpotent, normal  $P=C_p^2$ }                             \\
$a^q, b^p, c^p, b^a / b^{\rho(p,q)}$                   &                               & $\myw{p - 1}{q}$                     \\
$a^q, b^p, c^p, (b^a, c^a) / (b, c)^{M(p, q, \sigma_q^k)}$\hspace*{8ex}   & $0 \le k \le \frac12(q - 1)$\hspace*{8ex}   & $\frac12(q + 1 - \mydelta{2}{q})\myw{p - 1}{q}$       \\ 
$a^q, b^p, c^p, (b^a, c^a) / (b, c)^{\Irr_2(p,q)}$     &                               & $(1 - \mydelta{2}{q})\myw{p + 1}{q}$ \\
\hline
\multicolumn{3}{l}{{\mycol}{\bf Cluster 3:} non-nilpotent, normal $P=C_{p^2}$ }                          \\
$a^q, b^{p^2}, b^a / b^{\rho(p^2,q)}$                  &                               & $\myw{p - 1}{q}$                    \\
\hline
\multicolumn{3}{l}{{\mycol}{\bf Cluster 4:} non-nilpotent, normal $Q$ with complement  $P=C_p^2$  }                                                   \\
$a^p, b^p, c^q, c^a / c^{\rho(q, p)}$                  &                               & $\myw{q - 1}{p}$                    \\
\hline
\multicolumn{3}{l}{{\mycol}{\bf Cluster 5:} non-nilpotent, normal $Q$ with complement $P=C_{p^2}$  }                                                   \\
$a^{p^2}, b^q, b^a / b^{\rho(q, p)}$                   &                               & $\myw{q - 1}{p}$                    \\
$a^{p^2}, b^q, b^a / b^{\rho(q, p^2)}$                 &                               & $\myw{q - 1}{p^2}$                  \\
\bottomrule
\caption{Groups of order $p^2q$.}\label{tablep2q}
\end{longtable}

\pagebreak

\enlargethispage{1ex}
\renewcommand\arraystretch{1.1}
\begin{longtable}{lcr} 
 \toprule
  \multicolumn{3}{c}{\bf Groups of order $p^3q$, using Notation \ref{not1}.}                                  \\
\toprule 
PC-relators                                                        & Parameters       & Number of groups                                \\ 
\hline
\multicolumn{3}{l}{{\mycol}{\bf Cluster 1:} nilpotent}                                                                         \\
$a^{p^3q}$                                                         &                  & $1$                                             \\
$a^{p^2}, b^{pq}$                                                  &                  & $1$                                             \\
$a^p, b^p, c^{pq}$                                                 &                  & $1$                                             \\ 
                                                           
$a^p, b^p, c^p, d^q, b^a / bc$                                     &                  & $1-\Delta_p^2$                                             \\
$a^p / c, b^p, c^p, d^q, b^a / bc$                                 &                  & $1-\Delta_p^2$                                             \\ 
                                                
$a^2, b^4, c^q, b^a / b^3$                                              &                  & $\Delta_p^2$                                             \\
$a^2 / b^2, b^4, c^q, b^a / b^3$                                        &                  & $\Delta_p^2$                                             \\ 
\hline
\multicolumn{3}{l}{{\mycol}{\bf Cluster 2:} non-nilpotent, normal $Q$ with complement $P =C_{p^3}$}                                    \\
$a^{p^3}, b^q, b^a / b^{\rho(q, p)}$                               &                  & $\myw{q - 1}{p}$                               \\
$a^{p^3}, b^q, b^a / b^{\rho(q, p^2)}$                             &                  & $\myw{q - 1}{p^2}$                             \\
$a^{p^3}, b^q, b^a / b^{\rho(q, p^3)}$                             &                  & $\myw{q - 1}{p^3}$                             \\ 
\hline
\multicolumn{3}{l}{{\mycol}{\bf Cluster 3:} non-nilpotent, normal $Q$ with complement $P=C_{p^2} \times C_p$ }                         \\
$a^{p^2}, b^p, c^q, c^b / c^{\rho(q, p)}$                          &                  & $\myw{q - 1}{p}$                               \\
$a^{p^2}, b^p, c^q, c^a / c^{\rho(q, p)}$                          &                  & $\myw{q - 1}{p}$                               \\
$a^{p^2}, b^p, c^q, c^a / c^{\rho(q, p^2)}$                        &                  & $\myw{q - 1}{p^2}$                             \\ 
\hline
\multicolumn{3}{l}{{\mycol}{\bf Cluster 4:} non-nilpotent, normal $Q$ with complement $P =C_p^3$ }                                      \\
$a^p, b^p, c^p, d^q, d^a / d^{\rho(q, p)}$                         &                  & $\myw{q - 1}{p}$                               \\ 
\hline
\multicolumn{3}{l}{{\mycol}{\bf Cluster 5:} non-nilpotent, normal $Q$ with complement $P=p_+^{1+2}$ or $P= D_8$}                                   \\
$a^p, b^p, c^p, d^q, c^a / bc, d^a / d^{\rho(q, p)}$               &                  & $\myw{q - 1}{p}$                               \\
$a^2, b^4, c^q, b^a / b^3, c^a / c^{-1}$                           &                  & $\mydelta{2}{p}$                               \\ 
\hline 
\multicolumn{3}{l}{{\mycol}{\bf Cluster 6:} non-nilpotent, normal $Q$ with complement $P=p_-^{1+2}$ or $P=Q_8$}                            \\
$a^p, b^{p^2}, c^q, b^a / b^{p+1}, c^a / c^{\rho(q, p, k)}$        & $k \in \Z_p^*$   & $(1-\Delta_p^2)(p - 1)\myw{q - 1}{p}$          \\ 
$a^p, b^{p^2}, c^q, b^a / b^{p+1}, c^b / c^{\rho(q, p)}$           &                  & $(1-\Delta_p^2)\myw{q - 1}{p}$                 \\
                                          
$a^2 / b^2, b^4, b^a / b^3, c^q, c^a / c^{- 1}$                     &                  & $\Delta_p^2$                                   \\
\hline
\myrcol
\multicolumn{3}{l}{{\bf Cluster 7:} non-nilpotent, normal $P = C_{p^3}$ }                                                                                                   \\
$a^q, b^{p^3}, b^a / b^{\rho(p^3, q)}$                                                          &                               & $\myw{p - 1}{q}$                                                    \\ 
\hline
\myrcol
\multicolumn{3}{l}{{\bf Cluster 8:} non-nilpotent,  normal $P =C_{p^2} \times C_p$ }                                                                                        \\
$a^q, b^{p^2}, c^p, c^a / c^{\rho(p, q)}$                                                       &                               & $\myw{p - 1}{q}$                                                    \\
$a^q, b^{p^2}, c^p, b^a / b^{\rho(p^2, q)}$                                                     &                               & $\myw{p - 1}{q}$                                                    \\
$a^q, b^{p^2}, c^p, b^a / b^{\rho(p^2, q)}, c^a / c^{\rho(p, q, k)}$                            & $k \in \Z_q^*$               & $(q - 1)\myw{p - 1}{q}$                                             \\ 
\hline
\myrcol
\multicolumn{3}{l}{{\bf Cluster 9:} non-nilpotent, normal $P = C_p^3$ }                                                                                                     \\
$a^q, b^p, c^p, d^p, b^a / b^{\rho(p, q)}$                                                      &                               & $\myw{p - 1}{q}$                                                    \\
$a^q, b^p, c^p, d^p, b^a / b^{\rho(p, q)}, c^a / c^{\rho(p, q, \sigma_q^k)}$                      & $0 \le k \le \frac12(q-1)$   & $\frac12(q + 1 - \Delta_q^2) \myw{p - 1}{q}$                                     \\
$a^q, b^p, c^p, d^p, b^a / b^{\rho(p, q)}, c^a / c^{\rho(p, q)}, d^a / d^{\rho(p, q, k)}$& $k \in \Z_q^*$           & $(q - 1)\myw{p - 1}{q}$            \\
$a^q, b^p, c^p, d^p, b^a / b^{\rho(p, q)}, c^a / c^{\rho(p, q, \sigma_q^{k})}, d^a / d^{\rho(p, q, \sigma_q^\ell)}$ & $(k, \ell) \in \mathcal{P}$       & $\frac{1}{6}(q^2 - 5q + 6 + 4 \myw{q-1}{3})\myw{p - 1}{q}$           \\
$a^q, b^p, c^p, d^p, (b^a, c^a) / (b, c)^{\Irr_2(p, q)}$                               &                               & $(1 - \mydelta{2}{q})\myw{p + 1}{q}$                                   \\
$a^q, b^p, c^p, d^p, (b^a, c^a, d^a) / (b, c, d)^{\Irr_3(p, q)}$                       &                               & $(1 - \mydelta{2}{q})(1 - \mydelta{3}{q})\myw{p ^2 + p + 1}{q}$           \\ 
\hline
\myrcol
\multicolumn{3}{l}{{\bf Cluster 10:} non-nilpotent, normal $P = p_+^{1+2}$ }                                                                                                 \\
$a^q, b^p, c^p, d^p, c^b / cd, b^a / b^{{\rho(p, q, q - 1)}}, c^a / c^{\rho(p, q)}$              &                              & $\myw{p - 1}{q}$                                                    \\
$a^q, b^p, c^p, d^p, c^b / cd, b^a / b^{\rho(p, q)}, d^a / d^{\rho(p, q)}$                       &                              & $\myw{p - 1}{q}$                                                    \\
$a^q, b^p, c^p, d^p, c^b / cd, b^a / b^{{\rho(p, q, k)}}, c^a / c^{{\rho(p, q, q+1-k)}}, d^a/d^{{\rho(p,q)}}$         & $2 \le k \le \frac12(q + 1)$ & $\frac12(q - 1 - \mydelta{2}{q})\myw{p - 1}{q}$                    \\
$a^q, b^p, c^p, d^p, c^b / cd, (b^a, c^a) / (b, c)^{\Irr_2(p, q)}$                               &                              & $(1 - \mydelta{2}{q})(1-\mydelta{2}{p})\myw{p + 1}{q}$                                   \\
\hline
\myrcol
\multicolumn{3}{l}{{\bf Cluster 11:} non-nilpotent, normal $P = p_-^{1+2}$ or normal $P= Q_8$}                                                                                                 \\
$a^q, b^p, c^{p^2}, c^a / c^{\rho(p^2, q)}, c^b / c^{p+1}$.                                      &                              & $\myw{p - 1}{q}$                                                    \\
$a^3, b^2 / c^2, c^4, c^b / c^3, b^a / c, c^a / bc$                                              &                              & $\mydelta{2}{p}\mydelta{3}{q}$                                            \\ 
\hline
\myrcol
\multicolumn{3}{l}{{\bf Cluster 12:} no normal Sylow subgroups}  \\
$a^2, b^3, c^2, d^2, b^a / b^2, c^a / d, c^b / d, d^a / c, d^b / cd$                             &                               & $\myw{p}{2}\myw{q}{3}$\\
\bottomrule
\myrcol              {\bf Parameter sets}                                                 &                              &                                                                     \\
\multicolumn{3}{l}{$\mathcal{P} =\begin{cases} \{(x, y) \in \Z_{q - 1}^2 \;:\; 1 \le x \le \frac13(q - 2),\;\; 2x \le y \le q - 2 - x\} & (q \equiv 2 \bmod 3)\\\{(x, y) \in \Z_{q - 1}^2 \;:\; 1 \le x \le \frac13(q - 1),\;\; 2x \le y \le q - 2 - x\} \cup \{(\frac13(q - 1), \frac23(q - 1))\} & (q \equiv 1 \bmod 3).
  \end{cases}$}    \\
\bottomrule\caption{Groups of order $p^3q$. }
\label{p3q_notQnormal}
\end{longtable}

\pagebreak

\renewcommand\arraystretch{1.1}

\vspace*{-2ex}

\begin{longtable}{lcr} 
  \toprule
  \multicolumn{3}{c}{\bf Groups of order $p^2q^2$ with $p>q$, using Notation \ref{not1}.}\\
\toprule  
PC-relators                                                        & Parameters                     & Number of groups                      \\ 
\hline
\myrcol
\multicolumn{3}{l}{{\bf Cluster 1:} nilpotent }                                                                                \\
$a^{p^2q^2}$                                                       &                                   & 1                                              \\
$a^p, b^{pq^2}$                                                    &                                   & 1                                              \\
$a^{p^2q}, b^q$                                                    &                                   & 1                                              \\
$a^{pq}, b^{pq}$                                                   &                                   & 1                                              \\ 
\hline
\myrcol
\multicolumn{3}{l}{{\bf Cluster 2:} non-nilpotent, normal $P=C_{p^2}$ with complement $Q=C_{q^2}$}                                                                                      \\
$a^{q^2}, b^{p^2}, b^a / b^{\rho(p^2, q)}$                                  &                                & $\myw{p - 1}{q}$                     \\
$a^{q^2}, b^{p^2}, b^a / b^{\rho(p^2, q^2)}$                                &                                & $\myw{p - 1}{q^2}$            \\
\hline
\myrcol
\multicolumn{3}{l}{{\bf Cluster 3:} non-nilpotent,  normal $P=C_{p^2}$ with complement $Q=C_{q}^2$}                                                                                        \\
$a^q, b^q, c^{p^2}, c^a / c^{\rho(p^2, q)}$                                 &                                & $\myw{p - 1}{q}$                     \\
\hline
\myrcol
\multicolumn{3}{l}{{\bf Cluster 4:} non-nilpotent, normal $P=C_{p}^2$ with complement $Q=C_{q^2}$, or $(p,q)=(3,2)$}                                                                                  \\
$a^{q^2}, b^p, c^p, b^a / b^{\rho(p, q)}$                                   &                                & $\myw{p - 1}{q}$                     \\
$a^{q^2}, b^p, c^p, (b^a, c^a) / (b, c)^{M(p, q, \sigma_q^k)}$ \hspace*{8ex}           & $0 \le k \le \frac12(q - 1 )$\hspace*{8ex}   & $\frac12(q + 1-\myw{q}{2})\myw{p - 1}{q}$    \\
$a^{q^2}, b^p, c^p, b^a / b^{\rho(p, q^2)}$                                 &                                & $\myw{p - 1}{q^2}$             \\
$a^{q^2}, b^p, c^p, (b^a, c^a) / (b, c)^{M(p, q^2, \sigma_{q^2}^k)}$      & $0 \le k \le \frac12(q^2 - q)$ & $\frac12(q^2 - q + 2)\myw{p - 1}{q^2}$\\
$a^{q^2}, b^p, c^p, (b^a, c^a) / (b, c)^{M(p, q^2, kq)}$                  & $k \in \Z_q^*$                 & $(q - 1)\myw{p - 1}{q^2}$             \\
$a^9, b^2, c^2, b^a / c, c^a / bc$                                 &                                   & $\myw{p}{3}\myw{q}{2}$                                \\
$a^{q^2}, b^p, c^p, (b^a, c^a) / (b, c)^{\Irr_2(p, q)}$                     &                                & $(1 - \myw{q}{2})\myw{p + 1}{q}$                      \\
$a^{q^2}, b^p, c^p, (b^a, c^a) / (b, c)^{\Irr_2(p, q^2)}$                   &                                & $\myw{p + 1}{q^2}$             \\
\hline
\myrcol
\multicolumn{3}{l}{{\bf Cluster 5:} non-nilpotent,   normal $P=C_{p}^2$ with complement $Q=C_{q}^2$, or $(p,q)=(3,2)$}                                                                               \\
$a^q, b^q, c^p, d^p, c^a / c^{\rho(p, q)}$                                  &                                & $\myw{p - 1}{q}$                      \\
$a^q, b^q, c^p, d^p, (c^a, d^a) / (c, d)^{M(p, q, \sigma_q^k)}$           & $0 \le k \le \frac12(q - 1)$   & $\frac12(q + 1-\myw{q}{2})\myw{p - 1}{q}$    \\
$a^q, b^q, c^p, d^p, c^a / c^{\rho(p, q)}, d^b / d^{\rho(p, q)}$            &                                & $\myw{p - 1}{q}$                      \\
$a^3, b^3, c^2, d^2, c^a / d, d^a / cd$                            &                                   & $\myw{p}{3}\myw{q}{2}$                                 \\
$a^q, b^q, c^p, d^p, (c^a, d^a) / (c, d)^{\Irr_2(p, q)}$                    &                                & $(1 - \myw{q}{2})\myw{p + 1}{q}$                      \\
\bottomrule\caption{Groups of order $p^2q^2$ with $p>q$.}\label{tabp2q2}
\end{longtable}

\vspace*{-3ex}

\enlargethispage{6ex} 

\renewcommand\arraystretch{1.1}

\begin{longtable}[ht]{lcr}
  \toprule
  \multicolumn{3}{c}{\bf Groups of order $p^2qr$ with $r>q$, using Notation \ref{not1}.}\\
\toprule  
PC-relators                                                                                                                                                & Parameters                    & Number of groups                                \\ 
\hline
\myrcol\multicolumn{3}{l}{{\bf Cluster 1:} $F = G$}                                      \\
$a^{p^2qr}$                        &                               & 1        \\
$a^p, b^{pqr}$            &                               & 1    \\
\hline
\myrcol\multicolumn{3}{l}{{\bf Cluster 2:} $|F| = r$ }                                                                                                                                                                                \\
$a^{p^2q}, b^r, b^a / b^{\rho(r, p^2q)}$                                                                                                                    &                               & $\myw{r- 1}{p^2q}$                          \\\hline
\myrcol\multicolumn{3}{l}{{\bf Cluster 3:} $|F| = qr$ }                                                                                                                                                                               \\
$a^{p^2}, b^q, c^r, b^a / b^{\rho(q, p^2)}$      &                               & $\myw{q - 1}{p^2}$                             \\
$a^{p^2}, b^q, c^r, b^a / b^{\rho(q, p^2)}, c^a / c^{\rho(r, p, k)}$    & $k \in \Z_p^*$                & $(p - 1)\myw{q- 1}{p^2}\myw{r - 1}{p}$         \\
$a^{p^2}, b^q, c^r, b^a / b^{\rho(q, p^2)}, c^a / c^{\rho(r, p^2, k)}$   & $k \in \Z^*_{p^2}$            & $(p^2 - p) \myw{r - 1}{p^2}\myw{q - 1}{p^2}$   \\
$a^{p^2}, b^q, c^r, c^a / c^{\rho(r, p^2)}$                                                                                                                 &                               & $\myw{r - 1}{p^2}$                             \\
$a^{p^2}, b^q, c^r, b^a / b^{\rho(q, p)}, c^a / c^{\rho(r, p^2, k)},$                                                                                       & $k \in \Z_p^*$                & $(p - 1)\myw{r- 1}{p^2}\myw{q- 1}{p}$          \\
$a^p, b^p, c^q, d^r, c^a / c^{\rho(q, p)}, d^b / d^{\rho(r, p)}$                                                                                            &                               & $\myw{q - 1}{p}\myw{r - 1}{p}$                 \\
\hline
\myrcol\multicolumn{3}{l}{{\bf Cluster 4:} $|F| = p^2$ }                                                                                                                                                                              \\
$a^{qr}, b^{p^2}, b^a / b^{\rho(p^2, qr)}$                                                                                                                  &                               & $\myw{p - 1}{qr}$                              \\
$a^q, b^r, c^p, d^p, c^a / c^{\rho(p, q)}, d^b / d^{\rho(p, r)}$                                                                                            &                               & $\myw{p - 1}{qr}$                              \\
$a^q, b^r, c^p, d^p, c^a / c^{\rho(p, q)}, c^b / c^{\rho(p, r)}$                                                                                            &                               & $\myw{p - 1}{qr}$                              \\
$a^q, b^r, c^p, d^p, (c^a, d^a) / (c, d)^{M(p, q, k)}, c^b / c^{\rho(p, r)}$                                                                              & $k \in \Z_q^*$                & $(q - 1)\myw{p - 1}{qr}$                       \\ 
$a^q, b^r, c^p, d^p, c^a / c^{\rho(p, q)}, (c^b, d^b) / (c, d)^{M(p, r, k)}$                                                                              & $k \in \Z_r^*$                & $(r - 1)\myw{p - 1}{qr}$                       \\
$a^q, b^r, c^p, d^p, (c^a, d^a) / (c, d)^{M(p, q, \sigma_q^k)}, (c^b, d^b) / (c, d)^{M(p, r, \sigma_r^\ell)}$                                           & $(k, \ell) \in \mathcal{P}_1$          & $\frac12(qr - q - r + 5 - 2\mydelta{2}{q})\myw{p - 1}{qr}$       \\
$a^2, b^r, c^p, d^p, b^a / b^{-1}, (c^a, d^a) / (d, c), (c^b, d^b) / (c, d)^{M(p, r, r - 1)}$                                                             &                               & $\mydelta{2}{q}\myw{p-1}{r}$                   \\
$a^{qr}, b^p, c^p, (b^a, c^a) / (b, c)^{\Irr_2(p, qr)}$                                                                                                     &                               & $(1 - \mydelta{2}{q})\myw{p+1}{qr}$            \\
$a^q, b^r, c^p, d^p, (c^a, d^a) / (c, d)^{M(p, q)}, (c^b, d^b) / (c, d)^{\Irr_2(p, r)}$                                                                   &                               & $\myw{p - 1}{q}\myw{p + 1}{r}$                 \\
$a^2, b^r, c^p, d^p, b^a / b^{-1}, (c^a, d^a) / (d, c), (c^b, d^b) / (c, d)^{\Irr_2(p, r)}$                                                                 &                               & $\mydelta{2}{q}\myw{p+1}{r}$                   \\
$a^q, b^r, c^p, d^p, (c^a, d^a) / (c, d)^{\Irr_2(p, q)}, (c^b, d^b) / (c, d)^{M(p, r)}$                                                                   &                               & $(1 - \mydelta{2}{q})\myw{p+1}{q}\myw{p-1}{r}$ \\
\hline
\myrcol\multicolumn{3}{l}{{\bf Cluster 5:} $|F| = p^2q$ }\\
$a^r, b^{p^2}, c^q, b^a / b^{\rho(p^2, r)}$                                                                                                                 &                               & $\myw{p - 1}{r}$                               \\
$a^r, b^p, c^p, d^q, b^a / b^{\rho(p, r)}$                                                                                                                  &                               & $\myw{p - 1}{r}$                               \\
$a^r, b^p, c^p, d^q, (b^a, c^a) / (b, c)^{M(p, r, \sigma_r^k)}$                                                                                             & $0 \le k \le \frac12(r - 1)$  & $\frac12(r+1)\myw{p - 1}{r}$                 \\
$a^r, b^p, c^p, d^q, (b^a, c^a) / (b, c)^{\Irr_2(p, r)}$                                                                                                    &                               & $\myw{p+1}{r}$                                 \\
\hline
\myrcol
\multicolumn{3}{l}{{\bf Cluster 6:} $|F| = p^2r$ }                                                                                                                                                                                \\
$a^q, b^r, c^{p^2}, b^a / b^{\rho(r, q)}$                                                                              &                              & $ \myw{r - 1}{q}$                                                   \\
$a^q, b^r, c^{p^2}, c^a / c^{\rho(p^2, q)}$                                                                            &                              & $ \myw{p - 1}{q}$                                                   \\
$a^q, b^r, c^{p^2}, b^a / b^{\rho(r, q, k)}, c^a / c^{\rho(p^2, q)}$                                                   & $k \in \Z_q^*$               & $(q - 1) \myw{r - 1}{q} \myw{p - 1}{q}$                             \\
$a^q, b^r, c^p, d^p, c^a / c^{\rho(p, q)}$                                                                             &                              & $\myw{p - 1}{q}$     \\
$a^q, b^r, c^p, d^p, (c^a, d^a) / (c, d)^{M(p, q, \sigma_q^k)}$                                                      & $0 \le k \le \frac12(q - 1)$ & $\frac12(q + 1 - \mydelta{2}{q}) \myw{p - 1}{q}$             \\
$a^q, b^r, c^p, d^p, (c^a, d^a) / (c, d)^{\Irr_2(p, q)}$                                                               &                              & $ \myw{p + 1}{q}{\color{red}(1-\mydelta{2}{q})}$                                                   \\
$a^q, b^r, c^p, d^p, b^a / b^{\rho(r, q)}$                                                                             &                              & $ \myw{r - 1}{q}$                                                   \\
$a^q, b^r, c^p, d^p, b^a / b^{\rho(r, q)}, c^a / c^{\rho(p, q, k)}$                                                    & $k \in \Z_q^*$               & $(q - 1) \myw{r - 1}{q} \myw{p - 1}{q}$                             \\
$a^q, b^r, c^p, d^p, b^a / b^{\rho(r, q, \sigma_q^\ell)}, (c^a, d^a) / (c, d)^{M(p, q, \sigma_q^k)}$                     & $(k, \ell) \in \mathcal{P}_2$    & $\frac12q(q - 1 - \mydelta{2}{q}) \myw{r - 1}{q} \myw{p-1}{q}$      \\
$a^2, b^r, c^p, d^p, b^a / b^{-1}, c^a / c^{-1}, d^a / d^{-1}$                                                         &                              & $\mydelta{2}{q}$                                                       \\
$a^q, b^r, c^p, d^p, b^a / b^{\rho(r, q)}, (c^a, d^a) / (c, d)^{(\Irr_2(p, q)^k)}$                                      & $1 \le k \le \frac12(q - 1)$ & $\frac12(q - 1 - \mydelta{2}{q}) \myw{r - 1}{q} \myw{p + 1}{q}$                      \\
\hline
\myrcol
\multicolumn{3}{l}{{\bf Cluster 7:} $|F| = pr$ }                                                                                                                                                                                  \\
$a^q, b^p, c^p, d^r, d^a / d^{\rho(r, q)}, d^b / d^{\rho(r, p)}$                                                        &                             & $ \myw{r - 1}{pq}              $                                    \\
$a^q, b^{p^2}, c^r, c^a / c^{\rho(r, q)}, c^b / c^{\rho(r, p)}$                                                         &                             & $ \myw{r - 1}{pq}$                                                  \\
$a^q, b^p, c^p, d^r, c^a / {\color{red}c}^{\rho(p, q)}, d^b / d^{\rho(r, p)}$                                                        &                             & $ \myw{r - 1}{p} \myw{p - 1}{q}$                                    \\
$a^q, b^p, c^p, d^r, c^a / {\color{red}c}^{\rho(p, q, k)}, d^a / d^{\rho(r, q)}, d^b / d^{\rho(r, p)}$                               & $k \in \Z_q^*$              & $(q - 1) \myw{r - 1}{pq} \myw{p - 1}{q}$                            \\
\hline
\myrcol
\multicolumn{3}{l}{{\bf Cluster 8:} $|F| = pqr$ }                                                                                                                                                                                 \\
$a^{p^2}, b^q, c^r, c^a / c^{\rho(r, p)}$                                                                               &                             & $ \myw{r - 1}{p}$                                                   \\
$a^{p^2}, b^q, c^r, b^a / b^{\rho(q, p)}$                                                                               &                             & $ \myw{q - 1}{p}$                                                   \\
$a^{p^2}, b^q, c^r, b^a / b^{\rho(q, p)}, c^a / c^{\rho(r, p, k)}$                                                      & $k \in \Z_p^*$              & $(p - 1) \myw{r - 1}{p} \myw{q - 1}{p}$                             \\
$a^p, b^p, c^q, d^r, d^a / d^{\rho(r, p)}$                                                                              &                             & $ \myw{r - 1}{p}$                                                   \\
$a^p, b^p, c^q, d^r, c^a / c^{\rho(q, p)}$                                                                              &                             & $ \myw{q - 1}{p}$                                                   \\
$a^p, b^p, c^q, d^r, c^a / c^{\rho(q, p)}, d^a / d^{\rho(r, p, k)}$                                                     & $k \in \Z_p^*$              & $(p - 1) \myw{r - 1}{p} \myw{q - 1}{p}$                             \\
\hline
\myrcol
\multicolumn{3}{l}{{\bf Cluster 9:} $F=1$}                                          \\
${\rm Alt}_5$ (not solvable) &  & $\myw{p}{2}\myw{q}{3}\myw{r}{5}$\\
\hline
\myrcol 
\multicolumn{3}{l}{{\bf Parameter sets}}     \\
\multicolumn{3}{l}{$\mathcal{P}_1 = \{ (x, y) \colon 0 \le x \le \frac12(q - 1),\; 0 \le y \le \frac12(r - 1)\} \;\cup\; \{(x, y) \colon 1 \le x \le \frac12(q - 3),\; \frac12(r+1) \le y \le {r - 2}\}$ }                                                      \\[1.5ex]
\multicolumn{3}{l}{$\mathcal{P}_2 = \{(x, 0) \colon 0 \le x \le \frac12(q - 3)\} \;\cup\; \{(x, y) \colon 0 \le x \le \frac12(q - 1),\; 1 \le y \le \frac12(q - 3)\}\;\cup\; \{(\frac12(q - 1), \frac12(q - 1))\}$}                                                                      \\
\multicolumn{3}{l}{\hspace*{4.63cm}$\cup \; \{(x, y) \colon 0 \le x \le \frac12(q - 3),\; \frac12(q-1) \le y \le q-2\}$ }\\
\bottomrule\caption{Groups of order $p^2qr$ with $r>q$.}\label{tabp2qr} 
\end{longtable}


}
    
\subsection{Algorithms} \label{secalgmr}   

\noindent Theorem \ref{thp2q}  yields an enumeration formula for the groups
of order $n \in \mathcal{O}$, and our proof of Theorem \ref{thp2q} directly translates to a construction algorithm for one or all groups of order $n$. If only one specific group is to be constructed, then the clusters and enumeration formulas listed in Tables~\ref{tablep2q}--\ref{tabp2qr} are used: If the group with ID\footnote{\label{ftID}From Section 2.3 onwards ``ID'' always refers to the ordering used in our GAP implementation \cite{xp};  due to different construction algorithms, our ID can be different to the ID used in the SmallGroups Library.} $(n,i)$ is requested, then the counting formulas in the right columns of the tables allow us to directly determine the $i$-th group presentation. Conversely, if a group $G$ of order $n$ is given, then its ID $(n,i)$ can be computed by, firstly, determining the cluster that contains $G$, and, secondly, deciding to which group in the cluster the given group is isomorphic to; this can be done by computing various invariants of $G$ and comparing these with the choices made in Notation \ref{not1}. In Section \ref{secconst} we give more details and exemplify these construction and identification processes. A GAP implementation of our algorithms is available online, see \cite{xp}; we comment on its performance  in Section~\ref{secalg}.

\enlargethispage{6ex}


\section{Counting results}\label{secprel}

\noindent The aim of this section is to provide proofs for the various parts of Theorem \ref{thp2q}. Many of the groups we consider are split extensions $G=H\ltimes S$ with $S$ being a Sylow subgroup. If $S\cong C_p^m$ is elementary abelian, then $H$ acts on $S$ as a subgroup of $\Aut(S)\cong \GL_m(p)$. We therefore start this section by discussing some preliminary enumeration results for split extensions and subgroups of linear~groups.

\subsection{Counting subgroups of linear groups} 
The next lemma is from   Short \cite[Theorems 2.3.2 \& 2.3.3]{Short}.
\begin{lemma} \label{irred}
There is an irreducible cyclic subgroup of order $m$ in $\GL_n(p)$ if and only if $m \mid (p^n - 1)$ and 
$m \nmid (p^d - 1)$ for each $d\in\{1,\ldots,n-1\}$; if such a subgroup exists, then it is 
unique up to conjugacy. 
\end{lemma}

We now determine the number of conjugacy classes of certain subgroups of $\GL_2(p)$ and $\GL_3(p)$.


\begin{proposition}
\label{linear2}
For a group $G$ and integer $m$ write $s_m(G)$ for the number of conjugacy classes of subgroups of order $m$ in $G$. If  $p$, $q$, and $r$ are distinct primes, then the following hold.
\begin{ithm}
\item
If $q=2$, then $s_q(\GL_2(p)) = 2$; if $q > 2$, then $s_q(\GL_2(p)) = \tfrac{q+3}{2} \w_{p-1}^q + \w_{p+1}^q$. If $H\leq \GL_2(p)$ has order $q$, then $|N_{\GL_2(p)}(H)/C_{\GL_2(p)}(H)| \leq 2$ with equality for $q>2$ and $\w_{p-1}^q+\w_{p+1}^q$ subgroup~classes.
\item We have $s_{q^2}(\GL_2(p)) = \w_{p-1}^q + \tfrac{q^2+q+2}{2} \w_{p-1}^{q^2} + \w_{p+1}^{q^2}$.
\item
If $q=2$, then $s_{qr}(\GL_2(p)) = \tfrac{3r+7}{2} \w_{p-1}^r + 2 \w_{p+1}^r$; if $r,q > 2$, then
\[ s_{qr}(\GL_2(p)) = \tfrac{qr + q + r + 5}{2} \w_{p-1}^{qr} + 
\w_{p^2-1}^{qr}(1-\w_{p-1}^{qr}).\]
\item If $q=2$, then $s_q(\GL_3(p))=3$; if $q\geq 3$, then 
  \[ s_q(\GL_3(p)) = \tfrac{q^2 + 4q + 9 + 4 \w_{q-1}^3}{6} \w_{p-1}^q   +  \w_{(p+1)(p^2+p+1)}^q(1-\w_{p-1}^q). \]
\end{ithm}
\end{proposition}
 
\begin{proof}
  Let $S\cong C_{p^2-1}$ be a Singer cycle in $\GL_2(p)$ and let $D\cong C_{p-1}^2$ be the subgroup of diagonal matrices. Let $U\leq \GL_2(p)$ be of cubefree order $m$ with $p\nmid m$. It follows from  \cite[Lemma 8]{DEi05} that every such $U$  is conjugate to a subgroup of $N=C_2\ltimes S$ (Singer normaliser) or  of $I=C_2\ltimes D$ (maximal imprimitive). Recall that $D\cap S=Z(\GL_2(p))$ and $I/D$ swaps the diagonal entries of elements in $D$. Suppose now that $U$ is reducible. Then, up to conjugacy, $U\leq D$; if $U$ is cyclic of prime power order $m$, then, up to conjugacy, $U=\langle \diag(a,a^\ell)\rangle$ for some $\ell\in\{0,\ldots,m-1\}$ where $a\in \Z_p^\ast$ has order $m$. Note that two such groups $\langle \diag(a,a^\ell)\rangle$ and $\langle\diag(a,a^k)\rangle$ are conjugate if and only if $\ell=k$ or there is $x\in\Z_m^\ast$ with $\ell=x$ and $k\equiv x^{-1}\bmod m$. Thus, the conjugation action partitions these subgroups in pairs, unless the parameter $\ell\in\Z_m^\ast$ has order dividing~$2$. We use these results freely in the proof below.

\begin{iprf}
\item  Every group of order $q$ is cyclic. There exists a (unique) class of irreducible subgroups of order $q$ if and only if $q\mid (p+1)$ and $q\nmid (p-1)$; this requires $q\ne 2$. Now suppose $q\mid (p-1)$ and let $a\in\Z_p^\ast$ be of order~$q$. If $q=2$, then there are two classes of reducible subgroups, generated by $\diag(-1,-1)$ and $\diag(-1,1)$, respectively. If $q>2$ and $\sigma\in\Z_q^\ast$ is a generator, then there are $1+(q+1)/2$ classes of reducible subgroups, generated by $\diag(a,1)$ and $\diag(a, a^{(\sigma^k)})$ with $k\in\{0,\ldots, (q - 1)/2\}$. For the last claim, write $G=\GL_2(p)$ and let $H\leq G$ be of order $q$. If $q=2$, then $N_G(H)=C_G(H)$, so let $q>2$. Up to conjugacy, if $H$ is irreducible, then $H\leq S$ is unique and $N_G(H)/C_G(H)\cong C_2$;  if $H$ is reducible, then  $H=\langle \diag(a, a^{-1})\rangle\leq D$ with $a\in\Z_p^\ast$ of order $q$ is the unique such subgroup with $|N_G(H)/C_G(H)|=2$.

\item Any non-cyclic subgroup of order $q^2$ is reducible and exists if $q\mid (p-1)$; it is unique up to conjugacy. A cyclic irreducible subgroup of order $q^2$ exists if and only if  $q^2 \mid (p^2-1)$ and $q^2 \nmid (p-1)$; this forces $q = 2$ and $4 \nmid (p-1)$, or $q > 2$ and $q^2 \mid (p+1)$. A cyclic reducible subgroup of order $q^2$ requires $q^2\mid (p-1)$. If $a\in\Z_p^\ast$ has order $q^2$ and $\sigma\in\Z_{q^2}^\ast$ is a generator, then there are $(q^2+q+2)/2$ classes of of these subgroups, generated by $\diag(a,a^x)$ with $x\in\Z_{q^2}\setminus \Z_{q^2}^\ast$ and  $\diag(a, a^{(\sigma^k)})$ with $k\in\{0,\ldots, q(q - 1)/2\}$.

\item   First, suppose $q,r>2$. By Lemma \ref{irred}, there is one class of irreducible cyclic subgroups of order $qr$ if and only if $qr\mid (p^2-1)$ but $qr\nmid (p-1)$; this is the unique subgroup of order $qr$ in $N$, which yields a summand $\Delta_{p^2-1}^{qr}(1-\Delta_{p-1}^{qr})$. Looking at reducible cyclic subgroups, we require $qr\mid (p-1)$, and it remains to consider generators $\diag(a,a^\ell)$ with $a\in\mathbb{Z}_p^\ast$ of order $qr$ and $\diag(b,c)$ with $b,c\in\mathbb{Z}_p^\ast$ of order $q$ and $r$, respectively. In the latter case, there is a unique such subgroup. In the former case, we have to consider $q+r-1$ non-units $\ell\in \Z_{qr}\setminus \Z_{qr}^\ast$ and $qr-r-q+1$ units $\ell\in\Z_{qr}^\ast$, including four elements of order dividing~$2$. Together, we obtain $1+4+(qr-q-r-3)/2+(q+r-1)=(qr+r+q+5)/2$ classes of cyclic groups. Since $2\nmid qr$, there is no non-cyclic subgroup of order $qr$; this proves the formula for $q,r>2$.

  Now assume that $r>q=2$. If $r\mid (p+1)$, then $N$ contains, up to conjugacy, two irreducible subgroups of order $2r$, one being cyclic, the other non-cyclic; this yields the summand $2\w_{p+1}^r$. Now suppose $r\mid (p-1)$. We deal with reducible subgroups as before. The main difference is that this time $\Z_{2r}^\ast$ is cyclic, so there are only $2$ elements of order dividing~$2$; in total, we obtain $(3r+5)/2$ classes of reducible subgroups if $r\mid (p-1)$. It remains to count subgroups of $I$ that are not reducible: a short argument shows that, up to conjugacy, there is a unique such subgroup, namely, $\langle s,\diag(a,a^{-1})\rangle$ where  $a\in\Z_p^\ast$ has order $r$ and $s$ is the permutation matrix of the transposition $(1,2)$. This gives an  additional summand $\w_{p-1}^r$.

\item If $q=2$, then  $\diag(1,1,-1)$, $\diag(1,-1,-1)$, and $\diag(-1,-1,-1)$ generate the three classes of subgroups of order $2$; now let $q>2$.  Up to conjugacy, every diagonalisable subgroup of order $q$ is generated by an  element of type $\diag(a, 1, 1)$, $\diag(a, a^k, 1)$, or $\diag(a, a^k, a^\ell)$, respectively, where $a \in \Z_p^\ast$ has order $q$ and $k, \ell \in \Z_q^\ast$. By a), there are $(q+3)/2$ classes of groups with generators of the first two types; for the last type, let $\sigma\in\Z_q^\ast$ be a primitive element and note that $\langle \diag(a, a^{(\sigma^k)}, a^{(\sigma^\ell)})\rangle$ and $\langle \diag(a, a^{(\sigma^x)}, a^{(\sigma^y)})\rangle$ with $k,\ell,x,y\in \Z_{q-1}$ are conjugate if and only if
  $\{x,y\}\in \{ \{-\ell,k-\ell\},\{-k,\ell-k\},\{k,\ell\}\}$. To get the number of classes for this type, we use the Cauchy-Frobenius Lemma \cite[Lemma 2.17]{handbook} to count the $G$-orbits in the set of parameters $(k,\ell)\in\Z_{q-1}^2$, where
  \[G=\left\{\mymat{1}{0}{0}{1}, \mymat{0}{1}{1}{0},\mymat{-1}{-1}{0}{1}, \mymat{-1}{-1}{1}{0}, \mymat{0}{1}{-1}{-1},\mymat{1}{0}{-1}{-1} \right\}\cong {\rm Sym}_3.\]
  The number of orbits  is determined as $(q^2+q+4\w_{q-1}^3)/6$. In total, this gives $(q^2+4q+9+4\w_{q-1}^3)/6$ classes of diagonal subgroups of order $q$.   Non-diagonalisable groups arise from irreducible subgroups in $\GL_2(p)$ or $\GL_3(p)$. In 
  the former case, a) yields $\Delta_{p+1}^q$ groups. In the latter  case, Lemma~\ref{irred} shows that there is a (unique) class if and only if $q\mid (p^3-1)$ but $q\nmid (p^2-1)$. In total, there are $\Delta_{(p+1)(p^2+p+1)}^q(1-\Delta_{p-1}^q)$ non-diagonalisable groups. 
\end{iprf}
\end{proof}

\subsection{Counting split extensions}
A group $G$ splits over a normal subgroup $N\unlhd G$ if there exists a subgroup $U\leq G$ with $G=UN$ and $U\cap N=1$; in this case we write $G=U\ltimes N$. If $|N|$ and $|G/N|$ are coprime, then the Schur-Zassenhaus Theorem \cite[(9.1.2)]{rob} implies that $G$ splits over $N$. Conversely, if $U$ and $N$ are groups and $\varphi\colon U\to\Aut(N)$ is a homomorphism, then the group $G=U\ltimes_\varphi N$ with underlying set $\{(u,n): u\in U, n\in N\}$ and multiplication $(u,n)(v,m)=(uv,n^{\varphi(v)}m)$ splits over  $\{(1,n):n\in N\}\cong N$ with complement $\{(u,1): u\in U\}\cong U$. For some special cases, we will have to determine all split extensions of $U$ by $N$ up to isomorphism; the results of this section will be useful for that.

Let $U$ and $N$ be groups. The direct product $\Aut(U) 
\times \Aut(N)$ acts on $\varphi\in \Hom(U,\Aut(N))$ via \[ \varphi^{(\alpha, \beta)} = \ol{\beta^{-1}} \circ \varphi \circ \alpha, \]
where $\ol{\beta}\in\Aut(N)$ is the conjugation by $\beta$ in $\Aut(N)$, so $\varphi^{(\alpha,\beta)}(u)= \beta\circ \varphi(\alpha(u))\circ \beta^{-1}$ for all $u\in U$. The stabiliser of  $\varphi \in  \Hom(U,\Aut(N))$ under this action is the group of \emph{compatible pairs}
\[\Comp(\varphi)=\{(\alpha,\beta)\in\Aut(U)\times\Aut(N) :  \varphi(\alpha(u))=\beta^{-1}\circ \varphi(u)\circ \beta\quad (\forall u\in U)\}.\]
If $N$ is a abelian, then $\Comp(\varphi)$ acts on the group of $2$-cocycles $Z^2_\varphi(U,N)$ via $\gamma^{(\alpha,\beta)}=\beta^{-1}\circ\gamma\circ(\alpha,\alpha)$, inducing an action on the cohomology group $H_\varphi^2(U,N)$, see \cite[p.\ 55]{handbook}.  Every extension of $U$ by $N$ is isomorphic to a group $E(\varphi,\gamma)$ with underlying set $U\times N$ and multiplication $(u,n)(v,m)=(uv,n^{\varphi(u)}m\gamma(u,v))$ for some $\varphi\in\Hom(U,\Aut(N))$ and $\gamma\in Z^2_\varphi(U,N)$, see \cite[Section 2.7.3]{handbook}. Two such extensions of $U$ by $N$ are \emph{strongly isomorphic} if every isomorphism maps the normal subgroup $N$ to itself. This holds, for example, if  $U$ and $N$ are solvable of coprime orders, because in any extension $N$ is a unique  Hall subgroup, see \cite[(9.1.7)]{rob}, or if $N$ maps onto the Fitting subgroup in each extension. We call $N$ a \emph{strong $U$-group} if every pair of isomorphic extensions of $U$ by $N$ is also strongly isomorphic.

\begin{proposition} \label{Compat}
Let $U$ be a group and let $N$ be a strong $U$-group that is abelian. Let  $\O$ be a set of representatives of the $\Aut(U) \times \Aut(N)$ orbits in  $\Hom(U,\Aut(N))$,
and for each $\varphi \in \O$ let $o_\varphi$ be the number of $\Comp(\varphi)$-orbits in $H^2_\varphi(U,N)$. There are $\sum\nolimits_{\varphi \in \O} o_\varphi$ isomorphism
types of extensions of $U$ by $N$.
\end{proposition}
\begin{proof}Suppose $\tau\colon E(\varphi,\gamma)\to E(\psi,\delta)$ is an isomorphism. Since $N$ is a strong $U$-group, there exist $(\alpha,\beta)\in\Aut(U)\times\Aut(N)$ and a map $\varepsilon\colon U\to N$ such that each $\tau((u,n))=(\alpha(u),\beta(n)\varepsilon(u))$. Comparing the images of $(1,n^{\varphi(u)})=(u^{-1},1)(1,n)(u,1)$ under $\tau$ shows that $\beta(n^{\varphi(u)})=\beta(n)^{\psi(\alpha(u))}$ for all $u\in U$ and $n\in N$, so $\varphi$ and $\psi$ are in the same $\Aut(U)\times \Aut(N)$ orbit. This shows that extensions using different $\varphi,\psi\in\mathcal{O}$ are non-isomorphic.
Now consider  $\varphi\in \Hom(U,\Aut(N))$ with representative $\psi\in\O$, that is, there is $(\alpha,\beta)\in\Aut(U)\times\Aut(N)$ such that $\beta(n^{\varphi(u)})=\beta(n)^{\psi(\alpha(u))}$ for all $u\in U$ and $n\in N$; a direct computation shows that for every $\gamma\in Z^2_\varphi(U,N)$ the group $E(\varphi,\gamma)$ is isomorphic to $E(\psi,\gamma^{(\alpha^{-1},\beta^{-1})})$ via $(u,n)\mapsto (\alpha(u),\beta(n))$. The claim of the theorem now follows from  \cite[Section 4.2.1]{BEi99}, which shows that for a fixed $\varphi\in\mathcal{O}$ there are $o_\varphi$ isomorphism types of extensions $E(\varphi,\gamma)$ with $\gamma\in Z^2_\varphi(U,N)$. 
\end{proof}

A special case is considered in \cite[Lemma 5 and Theorem 14]{DEi05}:

\begin{lemma}\label{cubefree}
Let $U$ and $N$ be groups such that $N$ and the Sylow $p$-subgroup $P\leq U$ are cyclic of order $p$. If $N$ and $P$ are isomorphic as $N_U(P)$-modules, then there is a unique isomorphism type of non-split extensions of $U$ by $N$; otherwise every extension is a split extension.
\end{lemma}

We conclude this section by discussing the number of group extensions of solvable groups of coprime order; this requires the following  definition:

\begin{definition}
For solvable groups $U$ and $N$ of coprime order we define the following.  Let $\SS$ be a set of representatives for the conjugacy classes of subgroups in $\Aut(N)$,
let $\K$ be a set of representatives for the $\Aut(U)$-classes of 
normal subgroups in $U$, and set \[\O = \{ (S, K) : S \in \SS, K \in \K \mbox{ with } S \cong U/K \}.\] For $(S,K)\in\O$ let  $A_K, A_S\leq\Aut(U/K)$ be defined as follows:  $A_K$ is the subgroup induced by the action of $\Stab_{\Aut(U)}(K)$ on $U/K$; the group $A_S$ is defined as the preimage under a fixed isomorphism $U/K\to S$ of the subgroup of  $\Aut(S)$ induced by the action of $N_{\Aut(N)}(S)$. Finally, we set \[\ind_K = [\Aut(U/K) : A_K]\quad\text{and}\quad \DC(S, K) = A_K \setminus \Aut(U/K) /  A_S.\]
\end{definition}

\medskip

\begin{proposition} \label{Taunt}
  Let $N$ and $U$ be solvable of coprime orders. Let $\sigma(U,N)$ be the number of isomorphism types of extensions of $U$ by $N$. We have $\sigma(U,N)= \sum_{(S,K) \in \O} |\DC(S,K)|$, and the following hold:
  \begin{ithm}
  \item  If $N$ and $\Aut(N)$ are cyclic, then $\sigma(U,N)=\sum_{\ell \mid \pi} \sum_{K \in \K_\ell}  \ind_K$, where $\pi = \gcd(|U|,|\Aut(N)|)$ and $\K_\ell$ is a set of representatives of the 
$\Aut(U)$-classes in $\{K\unlhd U : U/K\cong C_\ell\}$.  
\item If $q$ is a prime and $U \cong C_{q^k}$, then $\sigma(U,N)=|\TT|$, where  $\TT$ is a set of conjugacy class representatives of cyclic subgroups of order dividing $q^k$ in $\Aut(N)$.
  \end{ithm} 
\end{proposition}

\begin{proof} 
  The assumptions imply that  $N$ is a strong $U$-group and that every extension splits; in particular, each $H^2_\varphi(U,N)=1$ is trivial.  Proposition \ref{Compat} (or Taunt \cite{Tau54}) shows that $\sigma(U,N)$ is the number of $\Aut(U) \times \Aut(N)$-orbits in $ \Hom(U,\Aut(N))$. These orbits correspond to the union of $A_K \times N_{\Aut(N)}(S)$-orbits in the set of isomorphisms $U/K \ra 
S$ for every $(S,K)\in\O$. The latter corresponds to  $\DC(S,K)$.
\begin{iprf}
\item For each $\ell \mid \pi$ there is a unique subgroup $S\leq \Aut(N)$ of order $\ell$. Since $\Aut(N)$ is abelian, $N_{\Aut(N)}(S)= \Aut(N)$  acts trivially on $S$ by conjugation, so $A_S=1$. This implies that $|\DC(S,K)| = \ind_K$ for each $K$.
\item For each $\ell=0,\ldots,k$ there is a unique $K\unlhd U$ with $|K|=p^\ell$ and  $U/K\cong C_{p^{k-\ell}}$. Each such $K$ satisfies $A_K = \Aut(U/K)$, hence $|\DC(S,K)| = 1$ in all cases.
 \end{iprf}
 \end{proof}


\subsection{Group enumeration}\label{secenum}

\noindent In the subsequent sections we  enumerate the isomorphism types of groups of order $p^2q$, $p^3q$, $p^2q^2$, and $p^2qr$, where $p,q,r$ are distinct primes. We freely use the following results in our discussion. Burnside's Theorem \cite[(8.5.4)]{rob} shows that groups of order $p^aq^b$ are solvable, and such groups have a nontrivial Fitting subgroup (which is the largest nilpotent normal subgroup). A group is nilpotent if and only if it is the direct product of its Sylow subgroups. Up to isomorphism, the groups of order $p$, $p^2$, and $p^3$ are the abelian groups $C_p$, $C_{p^2}$, $C_{p^3}$, $C_p^2$, $C_p^3$, $C_{p^2}\times C_p$, and two extraspecial groups $p_+^{1+2}$ and $p_-^{1+2}$; if $p=2$, then $p_+^{1+2}=D_8$ and $p_-^{1+2}=Q_8$. We often make case distinctions on the structure of the Fitting subgroup and on whether there exists a normal Sylow subgroup. Recall that if $G$ has a normal Sylow subgroup $N$, then $G=U\ltimes N$ for some $U\leq G$ by the Schur-Zassenhaus Theorem \cite[(9.1.2)]{rob}; in particular, the complement $U$ is unique up to conjugacy. The following lemma will be useful.

\begin{lemma}\label{lemFaith} 
If $G$ is a solvable group with abelian Fitting subgroup $F$, then $G/F$ acts faithfully on $F$ via conjugation; in particular, $G/F$ embeds into $\Aut(F)$ and so $|G/F|$ divides $|\Aut(F)|$.
\end{lemma}
\begin{proof}
  The assumptions imply that $F$ is nontrivial and $G$ acts via conjugation on $F$ with kernel $C_G(F)$. Since $G$ is solvable and $F$ is abelian, \cite[(5.4.4)]{rob} shows that $F=C_G(F)$. The claim follows.
\end{proof}

\subsection{The groups of order $p^2q$}\label{pf1}

The groups of order $p^2 q$ have been considered by H\"older 
\cite{Hol93}, by Cole \& Glover \cite{cole}, by Lin \cite{Lin74}, by Laue \cite{Lau82} and in various
other places. The results by H\"older, Lin and Laue agree with ours, although Lin's results have some harmless typos; we have not considered \cite{cole}. We also refer to 
\cite[Proposition 21.17]{BNV07} for an alternative description and proof of the following result.

\begin{proof}[Proof of Theorem \ref{thp2q}{\rm a)}]
  Let $G$ be a group of order $p^2q$. Since $G$ is solvable, the Fitting subgroup satisfies $1<F\leq G$. If $F=G$, then $G$ is nilpotent, so isomorphic to $C_{p^2} \times C_q$ or $C_p^2 \times C_q$. Now let $F<G$. The case $|F|=p$ and $|G/F|=pq$  contradicts Lemma \ref{lemFaith}, thus $|F|\in\{q,pq,p^2\}$ and $F$ contains a Sylow subgroup of $G$. The latter is characteristic in $F$, so normal in $G$. Thus, if $F<G$, then $G$ has a normal Sylow $p$- or Sylow $q$-subgroup, but not both. We make a case distinction and use Proposition~\ref{Taunt}.

If $G$ has a normal Sylow $p$-subgroup, then $G=U\ltimes N$ where $|N|=p^2$ and $U\cong C_q$ acts faithfully on $N$. If $N\cong C_p^2$, then $\Aut(N)\cong \GL_2(p)$ and the number of conjugacy classes of subgroups of $\Aut(N)$ of order $q$ is determined by Proposition \ref{linear2}a). If $N\cong C_{p^2}$, then $\Aut(N)\cong C_{p(p-1)}$ is cyclic and there are $\w_{p-1}^q$ subgroups of order $q$ in $\Aut(N)$. 
 this gives $\tfrac{q+3}{2}\w_{p-1}^q+\w_{p+1}^q+\w_{p-1}^q$ groups if $q > 2$, and $3$ groups if $q = 2$.

If $G$ has a normal Sylow $q$-subgroup, then $G=U\ltimes N$ where $|U|=p^2$ and $N\cong C_q$. To apply Proposition~\ref{Taunt}a) for counting non-abelian extensions, we have 
to consider the $\Aut(U)$-classes of proper normal subgroups $K$ in $U$ 
with $U/K$ cyclic. If $U$ is cyclic, then it has two proper normal subgroups $K$ with 
cyclic quotients; the case $K = 1$ arises if and only if $p^2 \mid (q-1)$ 
and the case $K = C_p$ arises if and only if $p \mid (q-1)$; in both cases 
$\ind_K = 1$, and together  we get $\w_{q-1}^p+\w_{q-1}^{p^2}$ groups. If $U$ is non-cyclic, then there is a unique $\Aut(U)$-class of proper normal subgroups 
$K$ with cyclic quotients, namely $K\cong C_p$ with $\ind_K = 1$; this case arises if $p \mid (q-1)$, and so we get $\w_{q-1}^p$ groups.  
\end{proof}

\subsection{The groups  of order $p^3q$}
Western \cite{Wes99} 
and Laue \cite{Lau82} determined these groups. Western's summary misses one group for $q \equiv 1 \bmod p$, but this group is mentioned in \cite[Section 13]{Wes99}; there are further minor issues in \cite[Section 32]{Wes99}. There are disagreements between our results and \cite[pp.~224-6]{Lau82} for the case $p=2$ and the case $q=3$; we have not tried to track the origin of these in Laue's work.

\begin{proof}[Proof of Theorem~\ref{thp3q}{\rm b)}]
  We follow the strategy of the proof of Theorem \ref{thp2q}a). Every group $G$ of order $p^3q$ is solvable and there are five nilpotent groups of order $p^3 q$; we now consider non-nilpotent groups.

  \medskip
  
 \noindent\emph{{\bf Case 1:} Non-nilpotent with normal Sylow $p$-subgroup.} Suppose $G=U\ltimes N$ with $|N|=p^3$ and $|U|=q$.  Proposition \ref{Taunt}b) shows that the isomorphism types of these groups correspond to the conjugacy classes of subgroups of order $q$ in $\Aut(N)$. We first consider the special case $p=2$: a direct computation shows that the only options are $N=C_2^3$ with $q\in\{3,7\}$ (one class each) and $N=Q_8$ with $q=3$ (one class). We now consider $p>2$ and discuss the different subgroups $N$.  

 If $N$ is cyclic, then $\Aut(N)\cong C_{p^2(p-1)}$ and there are $c_1=\w_{p-1}^q$ groups. If $N\cong C_{p^2}\times C_p$, then a short argument shows that $|\Aut(N)|=p^3(p-1)^2$; since $p>2$, the Sylow Theorem \cite[p.\ 40]{rob} implies that $\Aut(N)$ has a normal Sylow $p$-subgroup. Up to conjugacy, this Sylow $p$-subgroup has a unique complement  $C_{p-1}^2$, thus $\Aut(N)$ has  $c_2=(q+1)\w_{p-1}^q$  conjugacy classes of subgroups of order~$q$. If $N$ is extraspecial of exponent $p$, then there is an epimorphism $\Aut(N)\to \GL_2(p)$ with kernel $C_p^2$, see \cite[Theorem 1]{espec}, and so the conjugacy classes of subgroups $C_q$ in $\Aut(N)$ correspond to conjugacy classes of subgroups $C_q$ in $\GL_2(p)$. Proposition~\ref{linear2}a) shows that the number of such subgroups is $c_3=2$ if $q=2$, and $c_3=  \tfrac{1}{2}(q+3)\w_{p-1}^q + \w_{p+1}^q$ if $q>2$. If $N$ is extraspecial of exponent $p^2$, then \cite[Theorem~1]{espec} shows that $\Aut(N)$ has a normal Sylow $p$-subgroup and a complement isomorphic to $C_{p-1}$; this yields $c_4=\w_{p-1}^q$ groups.  Lastly, suppose $N\cong C_p^3$ is elementary abelian, so $\Aut(N)\cong \GL_3(p)$, and Proposition~\ref{linear2}d) yields $c_5=3$ groups if $q=2$; if $q>2$, then the number of groups we have to add is
\[ c_5 = \tfrac{1}{6}(q^2+4q+9 + 4\w_{q-1}^3) \w_{p-1}^q
+ \w_{(p+1)(p^2+p+1)}^q(1-\w_{p-1}^q).\]

\noindent\emph{{\bf Case 2:} Non-nilpotent with normal Sylow $q$-subgroup.} Suppose $G=U\ltimes N$ with $|N|=q$ and $|U|=p^3$.
Using Proposition
\ref{Taunt}, we have to find the $\Aut(U)$-orbits of proper normal 
subgroups $K\unlhd U$ with $U/K$ cyclic of order dividing $q-1$, and then 
for each such $K$ determine $\ind_K$. We make a case distinction on $U$.

If $U$ is cyclic, then the possibilities are $K\in\{1,C_p,C_{p^2}\}$ with each $\ind_K=1$; the number of groups is \[c_6=\w_{q-1}^p + \w_{q-1}^{p^2} + \w_{q-1}^{p^3}.\]

If $U\cong C_{p^2}\times C_p$, then there are two $\Aut(U)$-orbits of normal subgroups $K$ with $U/K \cong C_p$ and one $\Aut(U)$-orbit of $K\unlhd U$ with $U/K \cong C_{p^2}$; in each case, $\ind_K = 1$, so the number of groups is\[c_7 = 2 \w_{q-1}^p + \w_{q-1}^{p^2}.\]

If $U$ is extraspecial of exponent $p$ (or $U\cong Q_8$), then there is a unique $\Aut(U)$-orbit of normal subgroups $K$ with $U/K \cong C_p$; in this case, $\ind_K = 1$, and we have $c_8 = \w_{q-1}^p$ groups. If $U$ is extraspecial of exponent $p^2$ (or $U\cong D_8$), then there are two 
$\Aut(U)$-orbits of normal subgroups $K$ with $U/K \cong C_p$; one has 
$\ind_K = 1$ and the other has $\ind_K = p-1$, which yields $c_9 = p \w_{q-1}^p$ groups. Lastly, if $U\cong C_p^3$, then there is one $\Aut(U)$-orbit of $K\unlhd U$ with $U/K \cong C_p$; here $\ind_k = 1$, and we have $c_{10} = \w_{q-1}^p$ groups.

\medskip

\noindent\emph{{\bf Case 3:} Non-nilpotent with no normal Sylow subgroups.}
The Fitting subgroup $F$ satisfies $1<F<G$ and, as in the proof of Theorem \ref{thp2q}, we require that $F$ does not contain a Sylow subgroup of $G$. This forces $|F|=p$ or $|F|=p^2$, and Lemma \ref{lemFaith} yields $F\cong C_p^2$ as the only option. In this case $G/F$ has order $pq$ and embeds into $\Aut(F)\cong \GL_2(p)$, so $q\mid (p^2-1)$. If $G/F$ has a normal subgroup of order $p$, then its preimage in $G$ would be a normal Sylow subgroup, which is not possible. This shows that  $G/F\cong C_p\ltimes C_q$ and $p\mid (q-1)$. Since we also have $q\mid (p^2-1)$, we deduce $(p,q)=(2,3)$; this cases is treated separately.
\medskip

\noindent In summary, the cases $(p,q)\in\{(2,3),(2,7)\}$ are dealt with by a direct computation. For the other cases, it  remains to compute $5 + \sum_{i=6}^{10} c_i$ for $p = 2$ and $5 + \sum_{i=1}^{10} c_i$ for
$p > 2$; this yields the claimed formulas.
\end{proof}

\subsection{The groups of order $p^2q^2$}
The groups of order $p^2 q^2$ have been determined by Lin \cite{Lin74}, 
Le Vavasseur \cite{Vav02} and Laue \cite{Lau82}. Lin's work has only minor mistakes and is essentially 
correct; it agrees with our work and \cite[pp.~214-43]{Lau82}. Lin seems unaware of  \cite{Vav02}; we have not compared our results with those in \cite{Vav02}.

\begin{proof}[Proof of Theorem \ref{thp2q2}{\rm c)}]
  We assume $p>q$. Every group $G$ of order $p^2q^2$ is solvable and there are four nilpotent groups of order $p^2 q^2$. The case $(p,q)=(3,2)$ is dealt with by a direct calculation;  if  $(p,q)\ne (3,2)$, then the Sylow Theorem \cite[p.\ 40]{rob} implies that $G$ has a normal Sylow $p$-subgroup. Below we assume that  $G=U\ltimes N$ is non-nilpotent with $|N|=p^2$ and $|U|=q^2$. We use Proposition \ref{Taunt} below.

  First, let $N$ be cyclic, so $\Aut(N)\cong C_{p(p-1)}$. If $U$ is cyclic as well, then $\Aut(N)$ has at most subgroup of order $q$ and one subgroup of order $q^2$, respectively; this yields $c_1=\w_{p-1}^{q^2} + \w_{p-1}^q$ groups. If $U$ is not cyclic, then there is one $\Aut(U)$-orbit of $K\unlhd U$ with $U/K$ nontrivial cyclic; this yields $c_2=\w_{p-1}^q$ groups.

  Now let $N$ be non-cyclic and identify $\Aut(N)=\GL_2(p)$. For both possibilities of $U$, the $\Aut(U)$-orbits of proper subgroups of $U$ are $\mathcal{K}=\{1,C_q\}$. If $K\cong C_q$, then $U/K\cong C_q$ and $A_K=\Aut(U/K)$; it remains to count the number $c_3$ of conjugacy classes of subgroups of order $q$ in $\GL_2(p)$; Proposition~\ref{linear2}a) yields $c_3=2$ if $q=2$ and $c_3=\tfrac{q+3}{2}\w_{p-1}^q+\w_{p+1}^q$ otherwise. Note that $c_3$ has to be counted twice, for $U\cong C_{q^2}$ and for $U\cong C_q^2$. If $K=1$, then $A_K=\Aut(U/K)$ and it remains to count the number $c_4$ of conjugacy classes of subgroups of order $q^2$ in $\GL_2(p)$; this time we obtain $c_4=1+4\w_{p-1}^4+\w_{p+1}^4$ if $q=2$ and $c_4=\w_{p-1}^q+\tfrac{q^2+q+2}{2}\w_{p-1}^{q^2}+\w_{p+1}^{q^2}$ otherwise. The number of groups of order $p^2q^2$ now is $4+c_1+c_2+2c_3+c_4$. We use $\w_{p-1}^4=1-\w_{p+1}^4$ to rewrite the formula for $q=2$.
\end{proof}

\subsection{The groups of order $p^2 qr$}\label{pf2}

The groups of order $p^2 qr$ have been considered by Glenn \cite{Gle06} and
Laue \cite{Lau82}. Glenn's work has several problems: some groups are missing, there are duplicates, and some invariants are not correct. Laue \cite[p. 244-62]{Lau82} also does not agree with \cite{Gle06}; we have not 
compared our results with those of Laue.

\begin{proof}[Proof of Theorem \ref{thp2qr}{\rm d)}]
  Every non-solvable group has a non-abelian simple composition factor, which implies that the only non-solvable group of order $p^2qr$ is the alternating group $A_5$ of order $60$, see \cite[Theorem 2]{DEi05}. There are $c_0=2$ nilpotent groups of order $p^2qr$, and we now consider solvable non-nilpotent groups $G$ of order $p^2qr$. We make a case distinction on the Fitting subgroup $F\leq G$. By assumption, $1<F<G$, and Lemma~\ref{lemFaith} implies that $G/F$ embeds into $\Aut(F)$. We organise our case distinction by the number of prime factors of $|F|$; as before, we freely use Proposition \ref{Taunt}.

\smallskip
  
  \noindent\emph{{\bf Case 1:} $|F|$ is a prime.} If $|F|=p$, then $\Aut(F)=C_{p-1}$ and $G/F\cong C_{prq}$. This implies that $G$ has a normal (nilpotent) Sylow $p$-subgroup, a contradiction to $|F|=p$. If $|F|=q$, then $|G/F|=p^2r$ divides $q-1$, contradicting $r>q$. If $|F|=r$, then  $G\cong C_{p^2q}\ltimes F$, and there are $c_1=\w_{r-1}^{p^2q}$ such groups; note that $c_1=\w_{r-1}^{p^2}$ for $q=2$.

  \medskip

  \noindent\emph{{\bf Case 2:} $|F|$ is the product of two primes.}  If $F\cong C_{p^2}$, then $G\cong C_{qr}\ltimes F$ and $qr\mid (p-1)$; this gives $c_2=\w_{p-1}^{qr}$ groups, with $c_2=w_{p-1}^r$ if $q=2$. If $F\cong C_p^2$, then $G$ splits over $F$ and  $G/F$ can be considered as a subgroup of $\GL_2(p)$ of order $qr$. Using Proposition~\ref{linear2}c), the number of groups arising in this case is
\[c_3 =\begin{cases} \tfrac{3r+7}{2} \w_{p-1}^r + 2 \w_{p+1}^r &\text{(if $q=2$)}\\[0.7ex]
\tfrac{qr + q + r + 5}{2} \w_{p-1}^{qr} +  \w_{p^2-1}^{qr}(1-\w_{p-1}^{qr}) &\text{(otherwise.)}
\end{cases}\]
If $F\cong C_{pq}$, then $G/F\cong C_{pr}$ embeds into $\Aut(F)=C_{p-1}\times C_{q-1}$; since $r>q$, we have $r\mid (p-1)$ and $p\mid (q-1)$, but $r\leq p-1\leq q-2$ also contradicts $r>q$.

If $F\cong C_{pr}$, then $G/F\cong C_{pq}$, which forces $p\mid (r-1)$. Note that $G$ splits over the (normal) Sylow $r$-subgroup of $F$, so $G\cong \bar G\ltimes C_r$ where $\bar G$ has Fitting subgroup $\bar F\cong C_p$ with $\bar G/\bar F\cong C_{pq}$. There are two cases to consider. First, suppose that the Sylow $q$-subgroup of $G/F$ acts non-trivially on the Sylow $p$-subgroup of $F$. Then $q\mid (p-1)$ and $G\cong G/F \ltimes_\varphi F$ follows from Lemma \ref{cubefree} applied to $\bar F$ and $\bar G/\bar F$. As in Proposition \ref{Taunt}, the number of such groups is
given by the number of subgroups of order $pq$ in $\Aut(F)$. As the image 
of the Sylow $p$-subgroup of $G/F$ under $\varphi$ is uniquely determined, the number of groups is determined by the number of subgroups of order $q$ in $\Aut(F)\cong C_{p-1}\times C_{r-1}$ 
that act non-trivially on the Sylow $p$-subgroup of $F$; this number is $1 + (q-1) \w_{r-1}^q$. Second, suppose that the Sylow $q$-subgroup of $G/F$ acts trivially on the Sylow $p$-subgroup of $F$; then $q\mid (r-1)$ and the action of 
$G/F$ on $F$ is uniquely determined. By Lemma~\ref{cubefree}, there are two isomorphism types of extensions in this case. In summary, if $F\cong C_{pr}$, then the number of groups is
\[ c_4 =\begin{cases} 4\w_{r-1}^p &\text{(if $q=2$)}\\ \w_{r-1}^p( \w_{p-1}^q (1+ (q-1)\w_{r-1}^q) +2 \w_{r-1}^q )&\text{(otherwise.)}
\end{cases}\]

If $F\cong C_{qr}$, then $G=U\ltimes_\varphi F$ with $|U|=p^2$. We have to count the number of 
subgroups of order $p^2$ in $\Aut(F)\cong C_{q-1}\times C_{r-1}$. The number of subgroups  $C_p^2$ is $c_{5}' = \w_{q-1}^p \w_{r-1}^p$, and it remains determine the number of  subgroups $C_{p^2}$. This number depends on $\gcd(q-1, p^2) = p^a$ and 
$\gcd(r-1,p^2) = p^b$. If $a, b \leq 1$, then there is no subgroup $C_{p^2}$. If $(a,b)\in\{(2,0),(0,2)\}$, then there is one subgroup each; if $(a,b)\in\{(2,1),(1,2)\}$, then there are $p$ subgroups each; if $(a,b)=(2,2)$, then there are $p(p+1)$ subgroups. Together, the number of  cyclic subgroups $C_{p^2}$ is 
\[ c_5''= p(p+1)\w_{q-1}^{p^2}\w_{r-1}^{p^2} + \sum_{(u,v)\in\{(q,r),(r,q)\}} \left(\w_{u-1}^{p^2}(1-\w_{v-1}^p)+p \w_{u-1}^{p^2}\w_{v-1}^p(1-\w_{v-1}^{p^2})\right).\]
In summary, if $F\cong C_{qr}$, then we get $c_5=c_5'+c_5''$ groups, which can be written as  $c_5=  \w_{r-1}^{p^2}$ for $q=2$, and if $q>2$, then 
\[
c_5= (p^2 -p ) \w_{q-1}^{p^2} \w_{r-1}^{p^2} + (p-1) (\w_{q-1}^{p^2} \w_{r-1}^p + \w_{r-1}^{p^2} \w_{q-1}^p) + \w_{q-1}^{p^2} + \w_{r-1}^{p^2} + \w_{q-1}^p \w_{r-1}^p.
\]

 \noindent\emph{{\bf Case 3:} $|F|$ is the product of three primes.} If $F\cong C_{pqr}$, then $G= U\ltimes_\varphi N$ with $N\cong C_{qr}$ and $|U|=p^2$, and $\ker\varphi\cong C_p$. The group $U$ is either cyclic or elementary abelian, but in both cases there is a unique $\Aut(U)$-orbit of normal subgroups $K\unlhd U$ with $K\cong C_p$, and we have $A_K=\Aut(U/K)$; it remains to count the number of subgroups $C_p$ in $\Aut(N)\cong C_{q-1}\times C_{r-1}$. If $F\cong C_{pqr}$, then the number of groups is
 \[ c_6 = \begin{cases}
  2 \w_{r-1}^p &\text{(if $q=2$)}\\
   2( \w_{q-1}^p + \w_{r-1}^p + (p-1)\w_{q-1}^p\w_{r-1}^p) & \text{(otherwise.)}
 \end{cases}\]

If $|F|=p^2q$, then $G=U\ltimes_\varphi F$ with $F=N\times M$ with $|N|=p^2$, $M\cong C_q$, and $U\cong C_r$. We have to count the number of conjugacy class representatives of subgroups of order $r$ in 
$\Aut(F)\cong\Aut(N)\times C_{q-1}$. Since $r > q$, we need to count the number of conjugacy classes of subgroups of order
$r$ in $\Aut(N)$. If $N$ is cyclic, then $\Aut(N) \cong C_{p(p-1)}$ and we get $\w_{p-1}^r$ groups; if $N \cong C_p^2$, then $\Aut(N) \cong \GL_2(p)$ and Proposition~\ref{linear2}a) applies. In summary, if $|F|=p^2q$, then the number of groups is 
\[ c_7 = \tfrac{r+5}{2} \w_{p-1}^r + \w_{p+1}^r.\]

The case $|F|=p^2r$ is dual to the previous one, with the exception that now the bigger prime $r>q$ divides $|F|$. We have $G=U\ltimes_\varphi F$ and there are two possibilities for $F$. If $F\cong C_{p^2r}$, then $\Aut(F)$ is isomorphic to $C_{p(p-1)}\times C_{r-1}$ and we count the desired subgroups as
\[ c_{8}' = \w_{p-1}^q + \w_{r-1}^q + (q-1) \w_{p-1}^q\w_{r-1}^q. \]

The final case $F\cong C_p^2\times C_r$ is more complicated and required a slightly different approach. All groups in this case have the form $G=U\ltimes_\varphi N$ with $N\cong C_p^2$ and $|U|=qr$ such that $\varphi$ has kernel of order $r\mid |K|$, where $U/K$ is cyclic. There are two possibilities for $U$, either $U\cong C_{qr}$ or $U\cong C_q\ltimes C_r$ with $q\mid (r-1)$. If $U$ is cyclic, then $K\cong C_r$ (since $G$ is non-nilpotent); now  we  count the conjugacy classes of subgroups of order $q$ in $\Aut(N)$, which by Proposition~\ref{linear2} is determined as $c_8''=2$ if $q=2$ and $c_8''=\tfrac{q+3}{2}\w_{p-1}^q+\w_{p+1}^q$ otherwise. If $U$ is non-abelian, then $|K|\in\{r,qr\}$. If $K\cong C_{qr}$, then $\varphi$ is trivial hence there are $c_8'''=\w_{r-1}^q$ groups. It remains to consider $K\cong C_r$. Studying $\Aut(U)$ via the proof of Proposition \ref{Compat}, one can deduce that $A_K=1$. Thus, if $\SS$ is a set of conjugacy class representatives of 
subgroups of order $q$ in $\GL_2(p)$, then the number of groups arising in this case is
\[\w_{r-1}^q\sum\nolimits_{S \in \SS} |\Aut(U/K) / A_S|.\]
Here $A_S$ can be determined via Proposition~\ref{linear2}a), and since $\Aut(U/K)\cong C_{q-1}$, it follows that there are $c_8''''=2$ groups if $q=2$, and if $q>2$, then
\[
c_{8}'''' = \w_{r-1}^q\left(\tfrac{(q-1)(q+2)}{2} \w_{p-1}^q + \tfrac{q-1}{2} \w_{p+1}^q\right).
\]
In summary, this case adds $c_8 = c_{8}' + c_{8}'' + c_{8}'''+c_8''''$ groups, which is  $c_8 = 8$ if $q = 2$, and if $q > 2$ then 
\begin{eqnarray*}
c_8 &=&  \tfrac{(q-1)(q+4)}{2} \w_{p-1}^q \w_{r-1}^q 
        + \tfrac{q-1}{2} \w_{p+1}^q \w_{r-1}^q 
        + \tfrac{q+5}{2} \w_{p-1}^q  
        + 2 \w_{r-1}^q + \w_{p+1}^q.  
\end{eqnarray*}
Now $c_0+c_1+\ldots+c_8$ yields the formula for $p^2qr\ne 60$ given in the theorem.
\end{proof}

\section{Construction and identification}\label{secconst}
\noindent For $n\in\mathcal{O}$ let $\mathcal{G}_n$ be the list of all (isomorphism types of) groups of order $n$.

\subsection{Construction by ID} The proof Theorem \ref{thp2q} determines the size $\GN(n)$ of $\mathcal{G}_n$ by making various case distinctions on the structure of the groups $G$ of order $n$. Distinguishing features are to consider whether $G$ is nilpotent or non-nilpotent, whether $G$ has a normal Sylow subgroup, or whether the Fitting subgroup of $G$ has a specific structure. For each of these properties, the counting formulas developed in the proof of Theorem~\ref{thp2q} tell us exactly how many isomorphism types of these groups exist. This allows us to partition  $\mathcal{G}_n$ into various \emph{clusters} such that the groups within one cluster are split extensions that essentially only differ by having different action homomorphisms; it then remains to sort the different actions in a canonical way; this is explained in the proof of Theorem \ref{thp2q} and makes use of the canonical generators and matrices defined in Notation \ref{not1}. The formulas for the numbers of groups in each cluster (see the right columns in the tables) are the key ingredient that allows us to directly construct the $i$-th group in $\mathcal{G}_n$ \emph{without} constructing the whole list of group. A similar approach is used for the construction functionality provided by the SmallGroups library and by the algorithms in \cite{cgroups}; please cf.\ Footnote \ref{ftID}. Here we exemplify this approach by discussing the groups of order~$p^2q$; extensive details for all order types are given in the MPhil thesis \cite{xp} of the third~author.

\begin{example}\label{ex1}
  Let $n=p^2q$ with $q>2$. The proof of Theorem \ref{thp2q} partitions the groups in $\mathcal{G}_{n}$ in five cluster as given in Table \ref{tablep2q}. The groups in Cluster 1 are nilpotent and can be sorted by their exponents. Clusters~3 and 4 each contain at most one group, and so does Cluster 2 if $q\nmid (p-1)$. Cluster 5 contains at most two groups, and they can be distinguished by the order of the action group (the Sylow $p$-subgroup acting on the Sylow $q$-subgroup). If $q\mid (p-1)$, then Cluster 2 contains $(q+3)/2$ isomorphism types of groups and these are parametrised by the conjugacy classes of diagonalisable subgroups of $\GL_2(p)$ of order $q$, see the proof of Proposition~\ref{linear2}a). The latter proof also explains how to list these classes canonically: if $\sigma_p\in\Z_p^\ast$ and $\sigma_q\in\Z_q^\ast$ are the canonical generators and $a=\sigma_p^{(p-1)/q}$, then the subgroups we have to consider can be sorted as $\langle \diag(a,1)\rangle$ and $\langle M(p,q,\sigma_q^k))\rangle$ with $k\in\{0,\ldots, (q - 1)/2\}$, where each $M(p,q,\sigma_q^k)=\diag(a,a^{\sigma_q^k})$. For example, if $n=29^2.7$, then  Clusters 1--5 have $2$, $5$, $1$, $0$, $0$ groups, respectively, so the group with ID  $(29^2.7,6)$ is  $G=C_7\ltimes C_{29}^2$ where a generator $u\in C_7$ acts on  generators $v,w\in C_{29}^2$ via $M(29,7,\sigma_7^2)=\diag(a,a^{(\sigma_7^2)})$; here $\sigma_{29}=2$ and $\sigma_7=3$, so $a=\sigma_{29}^4=16$ and $a^{(\sigma_7^2)}=24$. Thus, $G$ can be defined via the presentation $\langle u,v,w \mid u^7,v^{29},w^{29},w^v/w, v^u/v^{16}, w^u/w^{24}\rangle$.
\end{example}

\subsection{Identification of groups} The identification of groups of order $n\in\mathcal{O}$ reverses the construction process: given a group $G$ of order $n$, we first determine to which cluster the group belongs by computing the Fitting and Sylow subgroups.  Then we determine the actions that arise. Lastly, it  remains to decide to which canonical action (namely, the one used in the construction-by-ID) this action is equivalent to.  We exemplify this process with another example and refer to  \cite{xp} for more information; the  main idea is to exploit the known structure of the relevant action groups described in the proof of Proposition~\ref{linear2}.

\begin{example}\label{ex2}
  Consider the group $G=\langle u,v,w \mid u^7,\;v^{29},\;w^{29},\;w^v/w,\; v^u/v^{24},\; w^u/v^{11}w^{7}\rangle$.  We determine $n=|G|=29^2.7$ and compute a Sylow $29$-subgroup $N$ and a Sylow $7$-subgroup $U$. We find  $N\unlhd G$, so $G$ is a split extension of $U$ by $N$. Since  $N\cong C_{29}^2$, we know that $G$ is a group in Cluster~2 as given in Table~\ref{tablep2q}. We choose generators $u\in U$ and $v,w\in N$ and observe that  $v^u=v^{24}$ and $w^u=v^{11}w^{7}$. (Here this can be determined directly from the  presentation, but if $G$ is given as a matrix or permutation group, then we would have likely chosen different generators.) Thus, $u$ acts on  $v,w$ via $\left(\begin{smallmatrix} 24&0\\11&7\end{smallmatrix}\right)\in\GL_2(29)$. This matrix has eigenvalues $\{24,7\}$, so it is conjugate to $\diag(24,7)=\diag(a^2,a^3)$ with $a=16$ as in Example \ref{ex1}. We have $\diag(a^2,a^3)^4=\diag(a,a^{(3^5)})=M(29,7,\sigma_7^5)$, but the parameter $5$ is greater than $(q-1)/2=3$; since  $3^{-5}=3$ in $\Z_7^\ast$, the proof of Proposition~\ref{linear2}a) shows that $\langle M(29,7,\sigma_7^5)\rangle$ is conjugate to $\langle M(29,7,\sigma_7)\rangle$. This determines the parameter $k=1$, and therefore $G$ has ID $(29^2.7,5)$.
\end{example}

\subsection{Implementation and Performance} \label{secalg}   
 
A GAP implementation of our algorithms  is  available in \cite{xp}.
Using this implementation we checked that our determination coincides (up to permuting the ordering of the lists of groups)
with that in the SmallGroups library, that obtained by the GrpConst 
package, see Besche-Eick-O'Brien \cite{BEO02}, and that obtained by the
Cubefree package, see Dietrich-Eick \cite{DEi05} and Dietrich-Wilson
\cite{cf}, for a large range of orders.

We now comment on the performance of our  algorithms;  all computations have been carried out with GAP 4.11.0 on a computer with  Intel(R) Core(TM) i5-7500 CPU@3.40GHz and 16GB RAM.\footnote{All our runtimes have been determined by using the option \url{USE_NC:=true} in our code: this avoids that GAP tests consistency of  polycyclic presentations, which becomes a major bottleneck when large primes are involved.} At the time of writing, the SmallGroups library contains the following orders discussed in this paper: $p^2q$ for all primes $p\ne q$, and $q^np$ for primes $p\ne q$ with  $q^n$ dividing one of $\{2^8,3^6,5^5,7^4\}$, and all relevant orders up to $2000$.

\begin{iprf}
\item[$\bullet$] There are 20514 groups of order $p^2q$ at most $10^5$. SmallGroups required 196 seconds to construct these groups, while our code took  16 seconds. Our code identified the groups constructed with SmallGroups in 78 seconds, whereas SmallGroups required 778 seconds to identify our groups. Up to order $10^6$, there are  159800 groups and our code required 120 seconds for the construction;  SmallGroups took 15181 seconds.  The reason for the increased runtime of SmallGroups is because for some order types the construction of groups involves some computations (akin to the ones described below), whereas our code directly writes down the group presentations.
\item[$\bullet$] There are 74562 groups of order $p^2q$, $p^3q$, $p^2q^2$ or $p^2qr$ at most 50000  available in the SmallGroups library. Our code required 47 seconds to construct these groups, whereas SmallGroups required 27359 seconds. Moreover, SmallGroups took 43356 seconds to identify our groups, while our code required 259 seconds to identify the groups constructed with SmallGroups.
\item[$\bullet$] Our code is also practical for larger primes; e.g.\ the construction of the $37371$, $6566$, and $21348$ groups of order $9341^3.467$, $127691^2.113^2$, and $415631^2.467.89$ took $32$, $5$, and $17$ seconds, respectively.\footnote{The performance of our code is even better ($5$, $1$, and $3$ seconds, respectively) if groups are returned as GAP objects {\tt pcp-group} (instead of {\tt pc-group}) by setting  \url{USE_PCP:=true}.} For such large primes we cannot compare our result with those of {\tt GrpConst} or {\tt Cubefree}, because the latter computations do not terminate in reasonable time (within a few hours). This is partly because these packages use general purpose algorithms which invoke computations with group homomorphisms and matrix groups. Our code avoids these bottlenecks by directly writing down polycyclic presentations of the (solvable) groups; the main bottleneck in our code seems to be GAP's pc-group arithmetic for large primes. 
\end{iprf}
We plan to extend the functionality of our implementation to other order types. For example, we will soon include construction and identification functionality for groups of order $p^4q$,  that is, our work makes most of the enumeration results of Eick-Moede  \cite{EM} constructive.

\def\cprime{$'$} \def\cprime{$'$}

\bibliographystyle{line}

\end{document}